\documentclass[10pt, reqno]{amsart}
\usepackage{amsaddr}
\usepackage{latexsym,amsfonts}
\usepackage{amssymb,amsthm,upref,amscd}
\usepackage[T1]{fontenc}
\usepackage{times}
\usepackage{amscd}
\usepackage{enumerate}
\usepackage{mathrsfs}
\usepackage{amsmath}
\usepackage{color}
\usepackage{soul}
\usepackage{indentfirst}
\usepackage{comment}
\PassOptionsToPackage{reqno}{amsmath}

\usepackage{mathtools}
\mathtoolsset{showonlyrefs}

\newtheorem{thm}{Theorem}[section]

\newtheorem{defi}{Definition}[section]
\newtheorem{lem}{Lemma}[section]
\newtheorem{cor}{Corollary}[section]

\newtheorem{rem}{Remark}[section]

\newcommand{\R}{\mathbb{R}}

\numberwithin{equation}{section}

\newcommand{\eps}{\epsilon}

\newcommand{\wto}{\rightharpoonup}

\makeatletter \@addtoreset{equation}{section} \makeatother

\usepackage[
left=30mm,
right=30mm,
top=25.4mm,
bottom=25.4mm,
headheight=2.17cm,
headsep=4mm,
footskip=12mm,
heightrounded,]{geometry}

\usepackage[colorlinks=true,urlcolor=blue,
citecolor=black,linkcolor=black,linktocpage,pdfpagelabels,
bookmarksnumbered,bookmarksopen]{hyperref}
\newcounter{const}
\author[V. Georgiev]{Vladimir Georgiev}
\address[V. Georgiev]{
\centerline{Dipartimento di Matematica Universit\'a di Pisa Largo B,}
\centerline{Pontecorvo 5, 56100 Pisa, Italy}
\centerline{Faculty of Science and Engineering, Waseda University 3-4-1,}
\centerline{Okubo, Shinjuku-ku, Tokyo 169-8555 Japan}
\centerline{Institute of Mathematics and Informatics at Bulgarian Academy of Sciences,}
\centerline{Acad. Georgi Bonchev Str., Block 8, 1113 Sofia}}

\author[T. Gou]{Tianxiang Gou}
\address[T. Gou]{
\centerline{School of Mathematics and Statistics, Xi’an Jiaotong University,}
\centerline{710049, Xi’an, Shaanxi, China}}



\email{georgiev@dm.unipi.it}
\email{tianxiang.gou@xjtu.edu.cn}
\title[Anisotropic 4NLS]{Solutions for fourth order anisotropic nonlinear Schr\"odinger equations in $\R^2$}
\thanks{{\it Acknowledgments}. V. Georgiev was partially supported by Gruppo Nazionale per l’Analisi Matematica project (2023/2024) "Modelli Nonlineari in presenza di interazioni puntuali", by the project PRIN 2020XB3EFL with the Italian Ministry of Universities and Research, by Institute of Mathematics and Informatics, Bulgarian Academy of Sciences, by Top Global University Project, Waseda University and the Project PRA 2022 85 of University of Pisa. T. Gou was supported by the National Natural Science Foundation of China (No. 12101483) and the Postdoctoral Science Foundation of China (No. 2021M702620). Part of work has been carried out when T. Gou visited Department of Mathematics, University of Pisa, whose hospitality he gratefully acknowledges.}
\thanks{{\it Conflict of interest statement}. The authors declare that there is no conflict of interests.}

\begin{document}

\begin{abstract}
In this paper, we consider solutions to the following fourth order anisotropic nonlinear Schr\"odinger equation in $\R \times \R^2$,
$$
\left\{
\begin{aligned}
&\textnormal{i}\partial_t\psi+\partial_{xx} \psi-\partial_{yyyy} \psi +|\psi|^{p-2} \psi=0, \\
&\psi(0)=\psi_0 \in H^{1,2}(\R^2),
\end{aligned}
\right.
$$
where $p>2$. First we prove the local/global well-posedness and blowup of solutions to the Cauchy problem for the anisotropic nonlinear Schr\"odinger equation. Then we establish the existence, axial symmetry, exponential decay and orbital stability/instability of standing waves to the anisotropic nonlinear Schr\"odinger equation. The pictures are considerably different from the ones for the isotropic nonlinear Schr\"odinger equations. The results are easily extendable to the higher dimensional case. 
\end{abstract}

\maketitle

\thispagestyle{empty}

\section{Introduction}

In this paper, we are concerned with solutions to the following fourth order anisotropic nonlinear Schr\"odinger equation in $\R \times \R^2$,
\begin{align}\label{equt}
\left\{
\begin{aligned}
&\textnormal{i}\partial_t\psi+\partial_{xx} \psi-\partial_{yyyy} \psi +|\psi|^{p-2} \psi=0, \\
&\psi(0, x)=\psi_0(x) \in H^{1,2}(\R^2),
\end{aligned}
\right.
\end{align}
where $p>2$. We denote by $H^{1,2}(\R^2)$ the underlying energy space, which is defined by the completion of $C^{\infty}_0(\R^2)$ under the norm
$$
\|u\|_{H^{1,2}}:=\|\partial_x u\|_2+\|\partial_{yy} u\|_2 +\|u\|_2.
$$
Higher order derivative modifications in nonlinear Schr\"odinger equations have been interpreted as higher order dispersion and they appear in the study of deep water waves, see for example \cite{D1, TKDV2000}. Equation  \eqref{equt} provides a canonical model for nonlinear Schr\"odinger equations with higher order anisotropic dispersion, which can be regarded as a variant of the models introduced in \cite{K91, K, KS} to describe stabilization of soliton instabilities of nonlinear Schr\"odinger equations with respect to small perturbations. It arises in nonlinear optics to describe the propagation of ultrashort laser pulses in medium with anomalous time-dispersion and it also arises in the study of propagation in fiber arrays and deep water wave dynamics, see for example \cite{ADRT, B, D, KH93, WF} and references therein. 

%On  the other hand, optical pulses in fibres are modeled also by NLS and \cite{T92}, \cite{KH93} suggest the study of third order dispersion effects.  The temporal modulation instabilities have been modelled and studied further in \cite{WF}.

%{\color{blue} This seems not convincing arises in nonlinear optics to model the propagation of ultrashort laser pulses in medium with anomalous time-dispersion and it also arises in the study of propagation in fiber arrays and deep water wave dynamics, see \cite{ADRT, B, D, WF} and references therein. }

Following the early works \cite{FIP, K, KS}, there are many researchers devoted to the study of solutions to the following fourth order nonlinear Schr\"odinger equation in $\R \times \R^n$,
\begin{align}\label{equt1}
\left\{
\begin{aligned}
&\textnormal{i}\partial_t\psi+\Delta \psi-\gamma \Delta^2 \psi +|\psi|^{p-2} \psi=0, \\
&\psi(0, x)=\psi_0(x) \in H^2(\R^n),
\end{aligned}
\right.
\end{align}
where $n \geq 1$, $\gamma \in \R$ and $p>2$.
%has been studied by many researchers in recent years.
%Equation \eqref{equt1} provides a canonical model for nonlinear Hamiltonian equations with dispersion of superquadratic order. 
For example, Ben-Artzi et al. in \cite{BKS}, Pausader in \cite{P2, P1, PS} and Pausader and Xia in \cite{PX} considered the well-posedness and scattering of solutions to \eqref{equt1}, see also \cite{D, Guo, MXZ} and references therein. Successively, Boulenger and Lenzmann in \cite{BL} proved the existence of blowup solutions to \eqref{equt1} for radially symmeytric data, which in turn confirmed a series of numerical studies conducted by Baruch et al. in \cite{BF, BFM, BFM1}. The blowup result in \cite{BL} has been recently extended to cylindrically symmetric data in \cite{Gou}. Furthermore, in \cite{BCdN, BCGJ1, BCGJ2}, Bonheure et al. investigated the existence and quantitative properties of standing waves to \eqref{equt1}. Here stand waves are solutions of the form 
\begin{align} \label{defsw}
\psi(t, x)=e^{\textnormal{i} \omega t} u(x), \quad w \in \R.
\end{align}
This clearly implies that the function $u$ solves the following fourth order equation,
\begin{align} \label{equ0}
\gamma \Delta^2 u- \Delta u +\omega u=|u|^{p-2} u \quad \mbox{in} \,\, \R^n.
\end{align}

It is well-known that, when $\gamma=0$, then solutions to \eqref{equt1} exist globally in time for $p<p_c$ or $p=p_c$ and $\|\psi_0\|_2<N_c$, where $N_c>0$ is a constant depending only on $n$ and
$
p_c=2+\frac 4n.
$ 
However, solutions to \eqref{equt1} may blow up for $p>p_c$ or $p=p_c$ and $\|\psi_0\|_2 \geq N_c$, see for example \cite{C, S}. While $\gamma \neq 0$, then solutions to \eqref{equt1} exist globally in time for $p<\widetilde{p}_c$ or $p=\widetilde{p}_c$ and $\|\psi_0\|_2<\widetilde{N}_c$, where $\widetilde{N}_c>0$ is a constant depending only on $n$ and
$
\widetilde{p}_c=2+\frac 8n.
$ 
However, solutions to \eqref{equt1} may blow up for $p>\widetilde{p}_c$ or $p=\widetilde{p}_c$ and $\|\psi_0\|_2 \geq \widetilde{N}_c$, see for example \cite{FIP}. This obviously shows that the higher order dispersive term helps to stabilize solutions, which in turn confirms the numerical results obtained in \cite{BF, BFM, BFM1, FIP}. It would be interesting to know whether higher order anisotropic dispersive term still has the same effect on stabilization of solutions. 
%To our knowledge, the consideration of solutions to fourth order anisotropic nonlinear Schr\"odinger equations as \eqref{equt} is open. 
As we shall see, when the dispersion is truly anisotropic in each direction, then the pictures to be presented below become considerably different.  %The anisotropy of \eqref{equ} will bring out new ingredients in the forthcoming study. 

The first aim of this paper is to consider the local well-posdeness of solutions to \eqref{equt}. In this direction, we have the following result.

\begin{thm} \label{lw}

Let $p>2$. Then, for any $R >0$,
%and $s \leq 1$, $s<(p-1)/2$, 
there exists a constant $T=T(R)>0$ such that 
%the integral equation
%\begin{equation}\label{equt07}
%      \psi(t) = e^{\textnormal{i} \left(-\partial_{xx} +\partial_{yyyy}\right) t}\psi_0  -\textnormal{i} \int_0^t e^{\textnormal{i}\left(-\partial_{xx} +\partial_{yyyy}\right)(t-\tau)} |\psi(\tau)|^{p-2} \psi(\tau) \,d \tau , 
%   \end{equation}
%associated with 
\eqref{equt} has a unique solution 
$$
\psi \in C([0,T); H^{1,2}(\mathbb{R}^2)) \cap C^1([0,T); H^{-1,-2}(\mathbb{R}^2)) 
$$
%$ \psi \in C([0,T), H^{1,2s}(\mathbb{R}^2))$ 
provided initial datum $\|\psi_0\|_{H^{1,2}(\mathbb{R}^2)} \leq R$, where $H^{-1,-2}(\R^2)$ is the dual space of $H^{1,2}(\R^2)$.
Moreover, the following properties hold true,
\begin{itemize}
%\item[(i)] If $p \in ((2,3],$ then $$\psi(t) \in C([0,T), H^{1,2}(\mathbb{R}^2)) \cap C^1([0,T), H^{-1,-2}(\mathbb{R}^2)) $$
%is a weak solution to \eqref{equt}.
\item[(i)] The mass and energy are conserved, i.e. for any $t \in [0, T)$,
$$
M(\psi(t))=M(\psi_0), \quad E(\psi(t))=E(\psi_0),
$$
where
$$
M(u):=\int_{\R^2} |u|^2 \, dxdy,
$$
$$
E(u):=\frac 12 \int_{\R^2} |\partial_x u|^2 \, dxdy + \frac 12 \int_{\R^2} |\partial_{yy} u|^2 \, dxdy -\frac 1p \int_{\R^2} |u|^p \, dxdy.
$$
%\item[(ii)] 
%The solution map $\psi_0 \in H^{1,2}(\mathbb{R}^2) \mapsto \psi  \in C([0,T); H^{1,2}(\mathbb{R}^2))$ 
%Lipschitz continuous in the ball of radius $R$ in $H^{1,2s}(\mathbb{R}^2)$. 
%is continuous in the ball of radius $R$ in $H^{1,2}(\mathbb{R}^2)$. 
%If $p>3$, then the solution map $\psi_0 \in H^{1,2}(\mathbb{R}^2) \mapsto \psi  \in C([0,T), H^{1,2}(\mathbb{R}^2))$ is Lipschitz continuous in the ball of radius $R$ in $H^{1,2}(\mathbb{R}^2)$.
%If $p \in (2,3],$ then the solution map $\psi_0 \in H^{1,2}(\mathbb{R}^2) \mapsto \psi  \in C([0,T), H^{1,2}(\mathbb{R}^2)) $is continuous in the ball of radius $R$ in $H^{1,2}(\mathbb{R}^2).$
\item[(ii)] Blow-up alternative holds: either $T= +\infty$ or $T< +\infty$ and $\|\psi(t)\|_{H^{1,2}}\to +\infty$ as $t \to T^-$.
\end{itemize}
\end{thm}

Since the problem under our consideration is anisotropic in two dimensions, then the standard ways based on the energy and Strichartz type estimates jointly with the contraction principle are not available to get the existence of solutions for $2<p<3$. Therefore, to prove Theorem \ref{lw}, 
%when $p>3$, we shall make use of the Strichartz estimates for the anisotropic operator and the contraction mapping principle to obtain the desired results. While $2<p<3$, 
we shall adapt \cite[Theorem 2.2]{OSY12} due to Okazawa, Suzuki and Yokota to demonstrate the existence of weak solutions to \eqref{equt}, which actually relies on the compactness arguments. Then, as an application of the Strichartz type estimates, we are able to establish the uniqueness of the solutions with the desirable regularity. 
%In view of \cite[Theorem 2.3]{OSY12}, then the proof is completed.
%The proof of Theorem \ref{lw} is mainly based on smoothing estimates obtained in \cite[Theorem 3.1]{KPV} and contraction mapping principle.

%{\color{red} 
%\begin{rem}
%The classical NLS is locally well posed in $L^2$ provided $p$ is mass cub critical. The mass critical value of $p$ in our case is $14/3$ (see Remark \ref{rem68} below). In the case $p \in (1, 14/3)$ one can show the local well posedness in $L^2$ too.  
%\end{rem}
%}

Next we are going to investigate standing waves to \eqref{equt}, which are solutions given by \eqref{defsw}.
%, which are solutions of the form 
%$$
%\psi(t, x)=e^{\textnormal{i} \omega t} u(x), \quad w \in \R.
%$$
It then follows that the function $u$ satisfies the following fourth order anisotropic equation,
\begin{align} \label{equ}
-\partial_{xx} u+ \partial_{yyyy} u +\omega u=|u|^{p-2} u \quad \mbox{in} \,\, \R^2.
\end{align}
%The problem under consideration arises from the study of standing waves of the time dependent nonlinear Schr\"odinger equation
%\begin{align}\label{equt}
%\textnormal{i}\partial_t \phi +\partial_{xx} \phi- \partial_{yyyy} \phi+|\phi|^{p-2} \phi=0 \quad \mbox{in} \,\, \R\times \R^2,
%\end{align}
%where $p>2$.
%Obviously, $H^{1,2}(\R^2)$ is a Hilbert space equipped with inner product
%$$
%(u ,v)_{H^{1,2}}:=(\partial_x u, \partial_v)_2 + (\partial_{yy} u, \partial_{yy} v)_2 + (u ,v)_2.
%$$
From variational perspectives, any solution to \eqref{equ} corresponds to a critical point of the underlying energy functional $J_{\omega}: H^{1,2}(\R^2) \to \R$ defined by
$$
J_{\omega}(u):=\frac 12 \int_{\R^2} |\partial_x u|^2 \, dxdy + \frac {\omega}{2} \int_{\R^2} |\partial_{yy} u|^2 \, dxdy+\frac 12 \int_{\R^2} |u|^2 \, dxdy -\frac 1p \int_{\R^2} |u|^p \, dxdy.
$$
Here we are only interested in ground states to \eqref{equ}, which are solutions minimizing the functional $J_{\omega}$ among all nontrivial solutions to \eqref{equ}. The existence of the solutions is addressed as follows.

%The first result in this direction reads as follows.

\begin{thm}\label{groundstate}
Let $p>2$ and $\omega>0$. Then there exist ground states to \eqref{equ}.
\end{thm}

To demonstrate the existence of ground states to \eqref{equ}, we shall introduce the following minimization problem with the help of the associated Nehari manifold,
\begin{align} \label{min1}
m_{\omega}:=\inf_{u \in N} J_{\omega}(u),
\end{align}
where 
$$
\quad N:=\left\{u \in H^{1,2}(\R^2) \backslash \{0\} : I_{\omega}(u)=0\right\},
$$
$$
I_{\omega}(u):=\int_{\R^2} |\partial_x u|^2 \,dxdy +\int_{\R^2} |\partial_{yy} u|^2 \,dxdy + \omega \int_{\R^2} |u|^2 \,dxdy -\int_{\R^2} |u|^p \, dxdy.
$$
Using the anisotropic Gagliardo-Nirenberg inequality and the associated concentration compactness result, see Lemmas \ref{ccl} and \ref{inequality}, we then have the existence of minimizers to \eqref{min1}. This in turn leads to the existence of ground states to \eqref{equ}.

\begin{rem}
Let $p>2$ and $\omega \leq 0$. Then there exists no solutions in $H^{1,2}(\R^2)$ to \eqref{equ}, see Corollary \ref{nonexistence}.
\end{rem}

In what follows, we are going to reveal some quantitative properties of ground states to \eqref{equ}. Due to the anisotropy of \eqref{equ}, then we cannot expect the radial symmetry of ground states. In fact, we have axial symmetry of the solutions.

\begin{thm} \label{symmetry}
Let $p>2$ and $\omega>0$. Then any ground state to \eqref{equ} is axially symmetric with respect to $x$-axis up to translations. Moreover, if $p>2$ and $p \in \mathbb{N}$, then any ground state to \eqref{equ} is axially symmetric with respect to $x$-axis and $y$-axis up to translations.
\end{thm}

Noting that the equation \eqref{equ} is anisotropic and has the higher-order dispersive term, to achieve the symmetry of the solutions with respect to $x$-axis, then we shall take advantage of the reflection techniques in \cite{L} and the unique continuation principle \cite[Theorem A.1]{dS}. When $p>2$ and $p \in \mathbb{N}$, to show the symmetry of the solutions with respect to $x$-axis and $y$-axis, we shall make use of the Fourier rearrangement arguments in \cite{BL, LS}.

\begin{thm} \label{decay}
Let $p>2$ and $\omega>0$. Let $u \in H^{1,2}(\R^2)$ be a solution to \eqref{equ}. Then $u \in L^{\infty}(\R^2) \cap C^{\infty}(\R^2)$. Moreover, there exist $C>0$ and $\sigma_0>0$ such that, for any $0<\sigma<\sigma_0$, 
$$
|u(x, y)| \leq C |y|^{-\frac 13} e^{-\sigma \left(|x|+|y|^{\frac 23}\right)}, \quad (x, y) \in \R^2.
$$
\end{thm}

\begin{rem}
Since \eqref{equ} is anisotropic, from Theorem \ref{decay}, then we see that solutions admit different decay rates at infinity in the $x$-axis and $y$-axis directions.
\end{rem}

We are now ready to discuss orbital stability and instability of ground states to \eqref{equ} in the following sense.

\begin{defi}
We say that a set $\mathcal{G} \subset H^{1,2}(\R^2)$ is orbitally stable if for any $\eps>0$, there exists $\delta>0$ such that if $\psi_0 \in H^{1,2}(\R^2)$ satisfies 
$$
\inf_{u \in \mathcal{G}}\|\psi_0-u\|_{H^{1,2}}<\delta,
$$ 
then the solution $\psi \in C(\R, H^{1,2}(\R^2))$ to \eqref{equt} with initial datum $\psi_0$ satisfies
$$
\sup_{t \in \R} \inf_{u \in \mathcal{G}} \|\psi(t) -u\|_{H^{1,2}}<\eps.
$$
Otherwise, we say that $\mathcal{G}$ is orbitally unstbale.
\end{defi}

\begin{defi}
For $u_{\omega} \in H^{1,2}(\R^2)$, we say that standing wave $e^{\textnormal{i} \omega t} u_{\omega}$ is orbitally stable if its orbit
$$
\mathcal{G}_{\omega} :=\left\{e^{\textnormal{i} \theta} u_{\omega}(\cdot + \xi_1. \cdot + \xi_2) : \theta \in \R, (\xi_1, \xi_2) \in \R^2\right\}
$$
is orbitally stable. Otherwise, we say that $e^{\textnormal{i} \omega t} u_{\omega}$ is orbitally unstable.
\end{defi}

\begin{thm} \label{stability}
Let $2<p<\frac{14}{3}$ and $\omega>0$. Then the set of ground states to \eqref{equ} is orbitally stable.
\end{thm}

First we observe that all ground states to \eqref{equ} admit the same $L^2$ norm. Then we introduce a global minimization problem given by the energy functional $E$ subject to the $L^2$-norm constraint. By applying the Lions concentration compactness principle in \cite{Li1, Li2}, we can derive the compactness of any minimizing sequence in $H^{1,2}(\R^2)$ up to translations. This then leads to the desired conclusion.

\begin{rem}
In view of the anisotropic Gagliardo-Nirenberg inequality in Lemma \ref{inequality}, we can clearly find that $p=\frac{14}{3}$ is the mass critical exponent for \eqref{equt}, $p<\frac{14}{3}$ and $p>\frac{14}{3}$ are referred to as the mass subcritical and supercritical exponents, respectively. It is well-known that $p=4$ and $p=6$ are the mass critical exponents for \eqref{equt1} in $\R^2$ with $\gamma=0$ and $\gamma \neq 0$, respectively.
\end{rem}

\begin{thm} \label{instability}
Let $p>\frac{14}{3}$ and $\omega>0$. Let $u_{\omega} \in H^{1,2}(\R^2)$ be a ground state to \eqref{equ}. Then stand wave $e^{\textnormal{i} \omega t} u_{\omega}$ is unstable.
\end{thm}

Since the non-degeneracy of ground states to \eqref{equ} is unkonwn, then we cannot adopt the elements in \cite{GSS} to establish Theorem \ref{instability}. In fact, the discussion of the non-degeneracy of ground states to \eqref{equ} is challengeable. Here, to discuss orbital instability of standing waves, we are inspired by the arguments in \cite{SS}, which depend mainly on the variational characterizations of ground states to \eqref{equ} instead of the spectral properties of the associated linearized operator. 
%In fact, the study in this aspect was inspired by \cite{BIK}. 

Finally, we plan to investigate dynamical behaviors of solutions to the dispersive equation \eqref{equt}, i.e. the global well-posedness and blowup of solutions to \eqref{equt}. Define
$$
Q(u):=\int_{\R^2} |\partial_x u|^2 \, dxdy + \int_{\R^2} |\partial_{yy} u|^2 \, dxdy -\frac{3(p-2)}{4p}\int_{\R^2} |u|^p \, dxdy.
$$
Here $Q(u)=0$ is Pohozaev's identity related to \eqref{equ}, see Lemma \ref{ph}. In this direction, the first result reads as follows, whose proof is based principally on the anisotropic Gagliardo-Nirenberg inequality in Lemma \ref{inequality}, the conservation laws and the variational characterizations of the ground states to \eqref{equ}. 

\begin{thm} \label{gw}
Let $p>2$ and $\psi \in C([0, T); H^{1,2}(\R^2))$ be the solution to the Cauchy problem for \eqref{equt} with initial datum $\psi_0 \in H^{1,2}(\R^2)$. Assume that one of the following conditions holds,
\begin{itemize}
\item [$(\textnormal{i})$] $2<p<\frac{14}{3}$.
\item [$(\textnormal{ii})$] $p=\frac{14}{3}$ and $\|\psi_0\|_2 < c_*$, where
$$
c_*:=\left(\frac{7}{3 C_{opt}} \right)^{\frac{3}{8}},
$$
$C_{opt}>0$ is the optimal constant for the anisotropic Gagliardo-Nirenberg inequality in Lemma \ref{gn}.
\item [$(\textnormal{iii})$] $p>\frac{14}{3}$ and $\psi_0 \in \mathcal{G}$, where
$$
\mathcal{G}:=\left\{ u \in H^{1,2}(\R^2) \backslash \{0\} : J_{\omega}(u) <m_{\omega}, Q(u)>0\right\}.
$$
\end{itemize}
Then $\psi(t)$ exists globally in time, i.e. $T=+\infty$.
\end{thm}

\begin{rem}
\begin{itemize}
\item[$(\textnormal{i})$] When $\gamma=0$ and $n=2$, then solutions to \eqref{equt1} exist globally in time for $p<4$ or $p=4$ and $\|\psi_0\|_2<N_c$. And solutions to \eqref{equt1} may blow up for for $p>4$ or $p=4$ and $\|\psi_0\|_2>N_c$. Theorem \ref{gw} justifies that the fourth order anisotropic dispersive term does help to stabilize solutions.
\item[$(\textnormal{ii})$] Since the higher order dispersion only occurs in one direction for \eqref{equt}, then the mass critical exponent is less than the one for \eqref{equt1} in $\R^2$ with $\gamma \neq 0$.
\end{itemize}
\end{rem}

Since \eqref{equ} is anisotropic, then there does not exist virial type quantities as introduced in the framework of the classical isotropic nonlinear Schr\"odinger equations to detect blowup of solutions to \eqref{equt}. In order to discuss blowup of solutions to \eqref{equt}, we need to introduce the so-called transverse virial quantity defined by
\begin{align}\label{v0}
V[u]:=\int_{\R^2} |x|^2 |u|^2 \,dxdy.
\end{align}
 To present the blowup result, we define
$$
K(u):=\frac 12 \int_{\R^2} |\partial_{yy} u|^2\, dxdy-\frac{p-2}{8p} \int_{\R^2}|u|^p\,dxdy.
$$
Here the functional $K$ appears in an alternative variational characterization of the ground states to \eqref{equ}. Indeed,  it is not hard to conclude that if $u \in H^{1,2}(\R^2)$ is a solution to \eqref{equ}, then $K(u)=0$. %, see Lemma \ref{vc0}.

\begin{thm} \label{blowup}
Let $p>\frac{14}{3}$ and $\psi \in C([0, T); H^{1,2}(\R^2))$ be the solution to the Cauchy problem for \eqref{equt} with initial datum $\psi_0 \in H^{1,2}(\R^2)$ satisfying $x u_0 \in L^2(\R^2)$. Assume that one of the following conditions holds,
\begin{itemize}
\item [$(\textnormal{i})$] $p \geq 6$ and $E(\psi_0) < 0$.
\item [$(\textnormal{ii})$] $p>10$ and $\psi_0 \in \mathcal{B}$ , where $\mathcal{B}:=\mathcal{B}_1 \cap \mathcal{B}_2$,
$$
\mathcal{B}_1:=\left\{ u \in H^{1,2}(\R^2) \backslash \{0\} : J_{\omega}(u) <m_{\omega}, Q(u)<0\right\},
$$
$$
\mathcal{B}_2:=\left\{ u \in H^{1,2}(\R^2) \backslash \{0\} : J_{\omega}(u) <m_{\omega}, K(u)>0\right\}.
$$
\end{itemize}
Then $\psi(t)$ blows up in finite time, i.e. $T<+\infty$.
\end{thm} 

\begin{rem}
\begin{itemize}
\item[$(\textnormal{i})$] The condition $p>10$ is required to establish an alternative variational characterizations of the ground states to \eqref{equ}. %Note that $p=6$ is the mass critical exponent for nonlinear Schr\"odinger equations in one dimension and $p=10$ is the mass critical exponent for fourth order nonlinear Schr\"odinger equations in one dimension.
\item[$(\textnormal{ii})$] To analysis blowup of solutions to the anisotropic nonlinear Schr\"odinger equations, we need to introduce the associated transverse virial quantity. Therefore, initial data belong to $\mathcal{B}_1$, which can no longer guarantee blowup of solutions. This is different from the study of blowup to solutions for the classical nonlinear Schr\"odinger equations.
\end{itemize} 
\end{rem}

The plan of the present paper is as follows. In Section \ref{pre}, we prove the anisotropic Gagliardo-Nirenberg inequality in $H^{1,2}(\R^2)$ and Pohozaev's identity satisfied by solutions to \eqref{equ}. In Section \ref{localw}, we study the local well-posedness of solutions to \eqref{equt} and prove Theorem \ref{lw}. In Section \ref{properties}, we discuss the existence and some properties of solutions to \eqref{equ} and establish Theorems \ref{groundstate}-\ref{instability}. Section \ref{dynamics} is devoted to the study of the global well-posedness and blowup of solutions to \eqref{equt}, which contains the proofs of Theorems \ref{gw} and \ref{blowup}.

To conclude this section, we shall introduce a scaling of $u \in H^{1,2}(\R^2)$ as
$$
u_\lambda(x)=\lambda^{\frac 38}u(\lambda^{\frac 12}x ,\lambda^{\frac 14} y), \quad \lambda>0.
$$
Such a scaling will be frequently used in the following discussion. It is straightforward to calculate that $\|u_{\lambda}\|_2=\|u\|_2$ and
$$
E(u_{\lambda})=\frac {\lambda}{2} \int_{\R^2} |\partial_x u|^2 \, dxdy + \frac {\lambda}{2} \int_{\R^2} |\partial_{yy} u|^2 \, dxdy-\frac {\lambda^\frac {3(p-2)}{8}}{p} \int_{\R^2} |u|^p \, dxdy.
$$
In addition, we see that

\begin{equation}\label{eq.Q93}
\begin{aligned}
\frac{d}{d \lambda} E(u_{\lambda})\mid_{\lambda=1}&=\frac {1}{2} \int_{\R^2} |\partial_x u|^2 \, dxdy + \frac {1}{2} \int_{\R^2} |\partial_{yy} u|^2 \, dxdy-\frac {3(p-2)}{8p} \int_{\R^2} |u|^p \, dxdy \\
&=\frac 12 Q(u).
\end{aligned}
\end{equation}

\begin{rem}
The mass critical exponent $p$ can be also defined by comparing the powers of the kinetic and potential energy in $E(u_\lambda)$, i.e. $1=3(p-2)/8$, which implies that $p=14/3$.   
\end{rem}

\section{Preliminary results} \label{pre}
%\begin{lem} 
%Let $q>2$. Then, for any $u \in H^1(\R)$,
%\begin{align} \label{gn1}
%\|u\|_q^q \leq C_{opt, 1} \|u\|_2^{\frac{q+2}{2}} \|\partial_x u\|_2^{\frac{q-2}{2}},
%\end{align}
%where $C_{opt, 1}>0$ is the optimal constant. % given by
%$$
%C_{opt, 1}=\frac{2q}{q-2} \frac{1}{\|Q\|_2^{\frac{q+2}{2}} \|\partial_x Q\|_2^{\frac{q-6}{2}}}
%$$
%$Q \in H^1(\R)$ is the ground state to the equation
%$$
%-\partial_{xx} Q+\omega Q=|Q|^{q-2} Q.
%$$
%\end{lem}

%\begin{lem} 
%Let $q>2$. Then, for any $u \in H^2(\R)$,
%\begin{align} \label{gn2}
%\|u\|_q^q \leq C_{opt, 2} \|u\|_2^{\frac{3q+2}{4}} \|\partial_{xx} u\|_2^{\frac{q-2}{4}},
%\end{align}
%where $C_{opt, 2}>0$ is the optimal constant.
%\end{lem}

In this section, we shall present some preliminary results used to establish the main results. First of all, we shall present the anisotropic Gagliardo-Nirenberg inequality in $H^{1,2}(\R^2)$. To prove this, we need the following concentration compactness lemma.

%\subsection{Anisotropic Gagliardo-Nirenberg inequality}

\begin{lem} \label{ccl}
Assume that $H^{1,2}(\R^2)$ is continuous embedded into $L^p(\R^2)$ for $p>2$. Let $\{u_n\} \subset H^{1,2}(\R^2)$ be a bounded sequence and 
$$
\sup_{z \in \R^2} \int_{B(z, R)} |u_n|^q \, dx=o_n(1).
$$ 
Then $u_n \to 0$ as $n \to \infty$ in $L^q(\R^N)$ for any $q>2$.
\end{lem}
\begin{proof}
Following the ideas of the proof of \cite[Lemma I.1]{Li2}, we can easily derive the desired conclusion. Then we omit the proof.
\end{proof}

\begin{lem} \label{inequality}
Let $p>2$. Then there exists $C_{opt}>0$ such that, for any $u \in H^{1,2}(\R^2)$,
\begin{align} \label{gn}
\|u\|_p^p \leq C_{opt} \|\partial_x u\|_2^{\frac{p-2}{2}} \|\partial_{yy} u\|_2^{\frac{p-2}{4}} \|u\|_2^{\frac{p+6}{4}}.
\end{align}
The optimal constant $C_{opt}>0$ is given by
$$
C_{opt}=\frac{p(p+6)^{\frac{3(p-2)}{8}-1}}{2^{\frac{p-2}{4}-2}(p-2)^{\frac {3(p-2)}{8}} \|W\|_2^{p-2}},
$$
where $W \in H^{1,2}(\R^2)$ is a ground state to the equation
$$
-\partial_{xx} W+ \partial_{yyyy} W + W=|W|^{p-2} W \quad \mbox{in} \,\, \R^2.
$$
\end{lem}
\begin{proof}
%In view of \eqref{gn1}, \eqref{gn2}, H\"older's inequality and Minkowski's inequality, then
%\begin{align*}
%\|u\|_p^p \leq C_{opt, 1}\int_{\R} \|u\|_{L^2_x}^{\frac{p+2}{2}} \|\partial_x u\|_{L^2_x}^{\frac{p-2}{2}} \, dy &\leq C_{opt, 1} \left(\int_{\R} \|u\|_{L^2_x}^{\frac{2(p+2)}{6-p}} \, dy\right)^{\frac{6-p}{4}}\|\partial_x u\|_2^{\frac{p-2}{2}} \\
%&=C_{opt ,1}\left\|\left\|u\right\|_{L^2_x}\right\|_{L^{\frac{2(p+2)}{6-p}}_y}^{\frac{p+2}{2}}\|\partial_x u\|_2^{\frac{p-2}{2}} \\
%& \leq C_{opt, 1}\left\|\left\|u\right\|_{L^{\frac{2(p+2)}{6-p}}_y}\right\|_{L^2_x}^{\frac{p+2}{2}}\|\partial_x u\|_2^{\frac{p-2}{2}} \\
%& \leq C_{opt, 1} C_{opt, 2}^{\frac{(p+2)^2}{6-p}}\left\|\left\|u\right\|_{L^2_y}^{\frac 34 +\frac{6-p}{4(p+2)}}\left\|\partial_{yy} u\right\|_{L^2_y}^{\frac 14 -\frac{6-p}{4(p+2)}}\right\|_{L^2_x}^{\frac{p+2}{2}}\|\partial_x u\|_2^{\frac{p-2}{2}} \\
%& \leq C_{opt, 1} C_{opt, 2}^{\frac{(p+2)^2}{6-p}} \|u\|_2^{\frac{p+6}{4}}\left\|\partial_{yy} u\right\|_2^{\frac{p-2}{4}}\|\partial_x u\|_2^{\frac{p-2}{2}}.
%\end{align*}
Define 
\begin{align} \label{wmin}
I_{opt}:=\inf_{u \in H^{1,2}(\R^2) \backslash \{0\}} I(u),
\end{align}
where
$$ 
I(u):=\frac{\|\partial_x u\|_2^{\frac{p-2}{2}} \|\partial_{yy} u\|_2^{\frac{p-2}{4}} \|u\|_2^{\frac{p+6}{4}}}{\|u\|_p^p}.
$$
%In view of \cite[Lemma 1]{FIS}, then $I_{opt}>0$. 
First we are going to demonstrate that $I_{opt}>0$. Let us begin with treating the case that $2<p \leq 4$. In this case, we need to introduce the following fractional Gagliardo-Nirenberg inequality in $H^{\alpha}(\R)$,
\begin{align} \label{fgn}
\|f\|_r^r \leq C \|(-\partial_{xx})^{\frac{\alpha}{2}} f\|_2^{\frac{r-2}{2\alpha}} \|f\|_2^{r-\frac{r-2}{2\alpha}}, \quad r>2.
\end{align}
Using \eqref{fgn} with $r=p$ and $\alpha=2$, we obtain that
$$
\|u(x, \cdot)\|_p^p \lesssim \|\partial_{yy} u(x, \cdot)\|_2^{\frac{p-2}{4}} \|u(x, \cdot)\|_2^{p-\frac{p-2}{4}}.
$$
Applying \eqref{fgn} with $r=\frac{2(3p+2)}{10-p}$ and $\alpha=1$ and H\"older's inequality, we then have that
\begin{align} \label{fgn1}
\begin{split}
\int_{\R^2} |u|^p \,dxdy &\lesssim \int_{\R} \|\partial_{yy} u(x, \cdot)\|_2^{\frac{p-2}{4}} \|u(x, \cdot)\|_2^{p-\frac{p-2}{4}}\,dx \\
& \leq \left(\int_{\R} \|u(x, \cdot)\|_2^{\frac{2(3p+2)}{10-p}} \,dy \right)^{\frac{10-p}{8}} \|\partial_{yy} u\|_2^{\frac{p-2}{4}} \\
& \lesssim \left\|\partial_{x}\left(\|u(x, \cdot)\|_2\right) \right\|^{\frac{p-2}{2}}_2 \|\partial_{yy} u\|_2^{\frac{p-2}{4}}  \|u\|_2^{\frac{p+6}{4}}.
\end{split}
\end{align}
In addition, we note that
$$
\partial_{x}\left(\|u(x, \cdot)\|_2\right)\|u(x, \cdot)\|_2=\frac 12 \partial_x\left(\|u(x, \cdot)\|_2^2\right) \leq \|\partial_x u(x, \cdot)\|_2 \|u(x, \cdot)\|_2.
$$
It clearly implies that
$$
\partial_{x}\left(\|u(x, \cdot)\|_2\right) \leq \|\partial_x u(x, \cdot)\|_2.
$$
Coming back to \eqref{fgn1}, we then get that
$$
\|u\|_p^p \lesssim \|\partial_x u\|_2^{\frac{p-2}{2}} \|\partial_{yy} u\|_2^{\frac{p-2}{4}} \|u\|_2^{\frac{p+6}{4}}.
$$
Let us now deal with the case $p>4$, which was inspired by \cite{FIS}. Observe that
\begin{align*}
|u|^{p_1}&=p_1 \int_{-\infty}^{x} |u|^{p_1-2} u\partial_x u \,dx \leq p_1 \int_{\R} |u|^{p_1-1} |\partial_x u| \,dx \\
& \leq p_1 \left(\int_{\R} |u|^{2(p_1-1)} \,dx\right)^{\frac 12} \left(\int_{\R} |\partial_x u|^2 \,dx \right)^{\frac 12}.
\end{align*}
In addition, we note that
\begin{align*}
|u|^{p_2} \leq &  p_2 \int_{\R} |u|^{p_2-1} |\partial_y u| \,dy=p_2 \int_{\R} |u|^{p_2-p_3-1} \left(|u|^{p_3}|\partial_y u| \right) \,dy \\
& \leq \frac{p_2}{\sqrt{2p_3+1}} \left(\int_{\R} |u|^{2(p_2-p_3-1)} \,dy\right)^{\frac 12}  \left(\int_{\R} \partial_{y} \left(|u|^{2p_3} u\right) \partial_y u \, dy\right)^{\frac 12} \\
& \leq \frac{p_2}{\sqrt{2p_3+1}}  \left(\int_{\R} |u|^{2(p_2-p_3-1)} \,dy\right)^{\frac 12}  \left(\int_{\R} |u|^{2p_3+1} |\partial_{yy} u| \, dy\right)^{\frac 12} \\
& \leq  \frac{p_2}{\sqrt{2p_3+1}}  \left(\int_{\R} |u|^{2(p_2-p_3-1)} \,dy\right)^{\frac 12} \left(\int_{\R} |u|^{2(2p_3+1)}\, dy\right)^{\frac 14} \left(\int_{\R}   |\partial_{yy} u|^2 \,dy \right)^{\frac 14}, 
\end{align*}
where $p_3>0$. Consequently, by H\"older's inequality, we have that
\begin{align} \label{fgn2}
\int_{\R^2} |u|^{p_1+p_2} \,dxdy \leq \frac{p_1p_2}{\sqrt{2p_3+1}}\|u\|_{2(p_2-p_3-1)}^{p_2-p_3-1} \|u\|_{2(2p_3+1)}^{\frac{2p_3+1}{2}}\|u\|_{2(p_1-1)}^{p_1-1} \|\partial_x u\|_2\|\partial_{yy} u\|_2^{\frac 12}.
\end{align}
Let 
$$
p_1=2(p_3+1), \quad p_2=3p_3+2, \quad p_3=\frac{p-4}{5}>0.
$$
It then follows that
$$
p_1+p_2=p, \quad 2(p_2-p_3-1)=2(2p_3+1)=2(p_1-1)=\frac{2(2p-3)}{5}>2.
$$
As a consequence, from \eqref{fgn2} and H\"older's inequality, we derive that
\begin{align*}
\int_{\R^2} |u|^p \,dx & \lesssim \|u\|_{\frac{2(2p-3)}{5}}^{\frac{2p-3}{2}}\|\partial_x u\|_2\|\partial_{yy} u\|_2^{\frac 12} \\
& \leq \|u\|_2^{\frac{(p+6)}{2(p-2)}}\|u\|_p^{\frac{p(p-4)}{p-2}}\|\partial_x u\|_2\|\partial_{yy} u\|_2^{\frac 12}.
\end{align*}
It immediately shows that
$$
\|u\|_p^p \lesssim \|\partial_x u\|_2^{\frac{p-2}{2}}\|\partial_{yy} u\|_2^{\frac {p-2}{4}}  \|u\|_2^{\frac{p+6}{4}}.
$$
Thereby, we have the desired conclusion. It indicates that $H^{1,2}(\R^2)$ is continuous embedded into $L^p(\R^2)$ for $p>2$.
%Let
%$$
%p_1+p_2=p, \quad 2(2p_3+1)=p, \quad 2(p_1-1)=p.
%$$
%It then follows that
%$$
%p_1=\frac p2 +1, \quad p_2=\frac p2-1, \quad p_3=\frac p 4-\frac 12.
%$$

In the followsing, we are going to show that $I_{opt}$ is achieved. For any $\mu, \lambda_1, \lambda_2 \in \R^+$ and $u \in H^{1,2}(\R^2)$, we define
\begin{align} \label{sc}
u_{\mu, \lambda_1, \lambda_2}(x, y):=\mu u(\lambda_1 x, \lambda_2 y).
\end{align}
It is simple to see that
$$
\|\partial_x u_{\mu, \lambda_1, \lambda_2}\|_2=\mu \lambda_1^{\frac 12} \lambda_2^{-\frac 12}\|\partial_x u\|_2, \quad \|\partial_{yy} u_{\mu, \lambda_1, \lambda_2}\|_2=\mu \lambda_1^{-\frac 12} \lambda_2^{\frac 32}\|\partial_{yy} u\|_2
$$
and
$$
\|u_{\mu, \lambda_1, \lambda_2}\|_2=\mu \lambda_1^{-\frac 12} \lambda_2^{-\frac 12}\|u\|_2, \quad \|u_{\mu, \lambda_1, \lambda_2}\|_p=\mu \lambda_1^{-\frac 1p} \lambda_2^{-\frac 1p}\|u\|_p.
$$
This immediately yields that $I(u_{\mu, \lambda_1, \lambda_2})=I(u)$. Let $\{u_n\} \subset H^{1,2}(\R^2) \backslash \{0\}$ be a minimizing sequence to \eqref{wmin}. Define $v_n:=(u_n)_{\mu_n, \lambda_{1,n}, \lambda_{2,n}}$, where
$$
\mu_n:=\frac{1}{\|\partial_x u_n\|_2^{\frac 12} \|\partial_{yy} u_n\|_2^{\frac 14}\|u_n\|_2^{\frac 14}}, \quad \lambda_{1,n}:= \frac{\|u_n\|_2}{\|\partial_x u_n\|_2}\quad \lambda_{2,n}:=\frac{\|u_n\|_2^{\frac 12}}{\|\partial_{yy}u_n\|_2^{\frac 12}}.
$$
As a consequence, we have that
\begin{align} \label{inf1}
\|\partial_x v_n\|_2=1, \quad \|\partial_{yy} v_n\|_2=1, \quad \|v_n\|_2=1
\end{align}
and
\begin{align} \label{inf2}
I(v_n)=\frac{1}{\|v_n\|_p^p}=I_{opt} +o_n(1).
\end{align}
Then $\{v_n\}$ is bounded in $H^{1,2}(\R^2)$. Due to $0<I_{opt}<+\infty$ and Lemma \ref{ccl}, then there exist a sequence $\{z_n\} \subset \R^2$ and a nontrivial $v \in H^{1,2}(\R^2)$ such that $v_n(\cdot+z_n) \wto v$ in $H^{1,2}(\R^2)$ as $n \to \infty$. Define $w_n:=v_n(\cdot+
z_n)$, then
$$
\|\partial_x w_n-\partial_x v\|_2^2+\|\partial_x v\|_2^2=\|\partial_x w_n\|_2^2+o_n(1),
$$
$$
\|\partial_{yy} w_n-\partial_{yy}v\|_2^2+\|\partial_{yy}v\|_2^2=\|\partial_{yy} w_n\|_2^2+o_n(1)
$$
and
$$
\|w_n-v\|_2^2+\|v\|_2^2=\|w_n\|_2^2+o_n(1), \quad \|w_n-v\|_p^p+\|v\|_p^p=\|w_n\|_p^p+o_n(1).
$$
Therefore, by using \eqref{inf1},  we have that
\begin{align*}
1+o_n(1)&=\|\partial_x w_n\|_2^{\frac{p-2}{4}} \|\partial_{yy} w_n\|^{\frac{p-2}{8}}\|w_n\|_2^{\frac{p+6}{8}}+o_n(1)\\
&=\left(\|\partial_x w_n-\partial_x v\|_2^2+\|\partial_x v\|_2^2+o_n(1)\right)^{\frac{p-2}{4}} \left(\|\partial_{yy} w_n-\partial_{yy}v\|_2^2+\|\partial_{yy}v\|_2^2+o_n(1)\right)^{\frac{p-2}{8}}\\
&\quad \times \left(\|w_n\|_2^2+\|v\|_2^2+o_n(1)\right)^{\frac{p+6}{8}} +o_n(1) \\
& \geq \|\partial_x w_n-\partial_x v\|_2^{\frac{p-2}{2}}\|\partial_{yy} w_n-\partial_{yy}v\|_2^{\frac{p-2}{4}}\|w_n-v\|_2^{\frac{p+6}{4}} +\|\partial_x v\|_2^{\frac{p-2}{2}}\|\partial_{yy}v\|_2^{\frac{p-2}{4}}\|v\|_2^{\frac{p+6}{4}} +o_n(1) \\
& \geq I_{opt}\|w_n-v\|_p^p+\|\partial_x v\|_2^{\frac{p-2}{2}}\|\partial_{yy}v\|_2^{\frac{p-2}{4}}\|v\|_2^{\frac{p+6}{4}} +o_n(1).
\end{align*}
On the other hand, by applying \eqref{inf2}, we derive that
\begin{align*}
I_{opt} \left(\|w_n-v\|_p^p+\|v\|_p^p+o_n(1)\right) +o_n(1)=I_{opt} \|v_n\|_p^p+o_n(1)=1+o_n(1).
\end{align*}
As a consequence, we obtain that 
$$
I_{opt} \geq \frac{\|\partial_x v\|_2^{\frac{p-2}{2}}\|\partial_{yy}v\|_2^{\frac{p-2}{4}}\|v\|_2^{\frac{p+6}{4}}}{\|v\|_p^p} +o_n(1).
$$
It then follows that $I_{opt}$ is attained by $v \in H^{1,2}(\R^2)$. Now we assume that $I_{opt}$ is achieved by $v \in H^{1,2}(\R^2)$ such that $\|\partial_x v\|_2=\|\partial_{yy} v\|_2=\|u\|_2=1$ and $\|v\|_p^p=C_{opt}$, where $C_{opt}:=1/I_{opt}$. Define
$$
\chi(t):=I(v+th), \quad t \in \R, \,\,\, h \in H^{1,2}(\R^2).
$$
Since $v \in H^{1,2}(\R^2)$ is a minimizer, then $\chi'(t)\mid_{t=0}=0$. This further gives that $v$ satisfies the equation
\begin{align} \label{infu}
-\frac{p-2}{2}\partial_{xx} v+ \frac {p-2}{4} \partial_{yyyy} v +\frac{p+6}{4} v=\frac{p} {C_{opt}} |v|^{p-2} v \quad \mbox{in} \,\, \R^2.
\end{align}
Note that $v$ is a minimizer, then it is a ground state to \eqref{infu}. Define $W:=\mu^{-1} u(\lambda_1^{-1} \cdot, \lambda_2^{-\frac 12} \cdot)$, 
%and $Q:= v(\lambda_1 \cdot, \lambda_2^{\frac 12} \cdot)$ for $\lambda_1, \lambda_2>0$, 
then $W \in H^{1,2}(\R^2)$ is a ground state to the equation
$$
-\frac{2(p-2) \lambda_1^{2}}{p+6}\partial_{xx} W+ \frac {(p-2) \lambda_2^2}{p+6} \partial_{yyyy} W + W=\frac{4\mu^{p-2}p}{C_{opt}(p+6)} |W|^{p-2} W.
$$
Let
$$
\frac{2(p-2) \lambda_1^{2}}{p+6}=1, \quad \frac {(p-2) \lambda_2^2}{p+6} =1, \quad \frac{4\mu^{p-2}p}{C_{opt}(p+6)}=1.
$$
Observe that
$$
1=\|u\|_2^2=\mu^2 \lambda_1^{-1} \lambda^{-\frac 12}_2 \|W\|_2^2.
$$
Then we can compute that
$$
C_{opt}=\frac{p(p+6)^{\frac{3(p-2)}{8}-1}}{2^{\frac{p-2}{4}-2}(p-2)^{\frac {3(p-2)}{8}} \|W\|_2^{p-2}}.
$$
This completes the proof.
\end{proof}

%\begin{cor}
%Let $p>2$. Then $H^{1,2}(\R^2)$ is continuously embedded into $L^p(\R^2)$ for any $p>2$.
%\end{cor}

\begin{lem} \label{ph}
Let $p>2$ and $u \in H^{1,2}(\R^2)$ be a solution to \eqref{equ}. Then $u$ satisfies the following Pohozaev's identity,
$$
\int_{\R^2} |\partial_x u|^2 \, dxdy + \int_{\R^2} |\partial_{yy} u|^2 \, dxdy =\frac{3(p-2)}{4p}\int_{\R^2} |u|^p \, dxdy.
$$
\end{lem}
\begin{proof}
%Since $u \in H^{1,2}(\R^2)$ is a solution to \eqref{equ}, by regularity theories, then $u\in H^{2,2}(\R^2)$, $x \partial_x u \in L^2(\R^2)$ and $y\partial_y u \in L^2(\R^2)$. 
Define $Q_R:=[-R, R] \times [-R, R]$ for $R>0$. First multiplying \eqref{equ} by $x \partial_x u, y \partial_y u$ and integrating over $Q_R$, we then derive the associated Pohozaev's identity on $Q_R$. Next taking the limit as $R \to \infty$, we then have the desired result. 
%Note that
%$$
%\int_{\R^2}f \, dxdy=\lim_{R \to + \infty} \int_{Q_R} f\, dxdy.
%$$ 

For simplicity, we shall directly integrate over $\R^2$. Multiplying \eqref{equ} by $x \partial_x u$ and integrating over $\R^2$, we then obtain that
\begin{align*} %\label{ph1}
-\int_{\R^2} \partial_{xx} u (x \partial_x u) \, dxdy+\int_{\R^2} \partial_{yyyy} u (x \partial_x u) \, dxdy + \omega \int_{\R^2} u (x \partial_x u)  \, dxdy=\int_{\R^2} |u|^{p-2}u (x \partial_x u) \, dxdy.
\end{align*}
Observe that
$$
-\int_{\R^2} \partial_{xx} u (x \partial_x u) \, dxdy=-\frac 12 \int_{\R^2} x \partial_x \left(|\partial_x u|^2\right)\, dxdy=\frac 12 \int_{\R^2}|\partial_x u|^2\, dxdy,
$$
\begin{align*}
\int_{\R^2} \partial_{yyyy} u (x \partial_x u) \, dxdy&=\int_{\R^2} \partial_{yy} u (x \partial_{xyy}u) \, dxdy =\frac 12 \int_{\R^2} x \partial_x \left(|\partial_{yy} u|^2\right) \, dxdy=-\frac 12 \int_{\R^2} |\partial_{yy} u|^2 \,dxdy,
\end{align*}
$$
\omega \int_{\R^2} u (x \partial_x u)  \, dxdy=\frac{\omega}{2} \int_{\R^2} x \partial_x \left(|u|^2\right) \, dxdy=-\frac{\omega}{2} \int_{\R^2} |u|^2 \, dxdy
$$
and
$$
\int_{\R^2} |u|^{p-2}u (x \partial_x u) \, dxdy=\frac 1 p \int_{\R^2} x \partial_x \left(|u|^p\right) \,dxdy=-\frac 1p \int_{\R^2} |u|^p \, dxdy.
$$
This leads to
\begin{align} \label{ph2}
\frac 12 \int_{\R^2}|\partial_x u|^2\, dxdy-\frac 12 \int_{\R^2} |\partial_{yy} u|^2 \,dxdy-\frac{\omega}{2} \int_{\R^2} |u|^2 \, dxdy=-\frac 1p \int_{\R^2} |u|^p \, dxdy.
\end{align}
On the other hand, multiplying \eqref{equ} by $y \partial_y u$ and integrating over $\R^2$, we then obtain that
\begin{align*} %\label{ph1}
-\int_{\R^2} \partial_{xx} u (y \partial_y u) \, dxdy+\int_{\R^2} \partial_{yyyy} u (y \partial_y u) \, dxdy + \omega \int_{\R^2} u (y \partial_y u)  \, dxdy=\int_{\R^2} |u|^{p-2}u (y \partial_y u) \, dxdy.
\end{align*}
Notice that
$$
-\int_{\R^2} \partial_{xx} u (y \partial_y u) \, dxdy=\int_{\R^2} \partial_x u (y \partial_{xy} u) \, dxdy=\frac 12 \int_{\R^2} y \partial_y \left(|\partial_x u|\right) \, dxdy=-\frac 12 \int_{\R^2} |\partial_x u|^2 \,dxdy
$$
and
$$
\int_{\R^2} \partial_{yyyy} u (y \partial_y u) \, dxdy=\int_{\R^2} \partial_{yy} u (y \partial_{yyy} u +2 \partial_{yy} u) \, dxdy=\frac 32 \int_{\R^2} |\partial_{yy} u|^2 \, dx.
$$
Therefore, we get that
\begin{align} \label{ph3}
-\frac 12 \int_{\R^2}|\partial_x u|^2\, dxdy+\frac 32 \int_{\R^2} |\partial_{yy} u|^2 \,dxdy-\frac{\omega}{2} \int_{\R^2} |u|^2 \, dxdy=-\frac 1p \int_{\R^2} |u|^p \, dxdy.
\end{align}
Next multiplying \eqref{equ} by $u$ and integrating over $\R^2$, we then find that
\begin{align} \label{ph4}
 \int_{\R^2}|\partial_x u|^2\, dxdy+\int_{\R^2} |\partial_{yy} u|^2 \,dxdy+\omega\int_{\R^2} |u|^2 \, dxdy=\int_{\R^2} |u|^p \, dxdy.
\end{align}
Combining \eqref{ph2}, \eqref{ph3} and \eqref{ph4}, we then derive that
$$
 \int_{\R^2}|\partial_x u|^2\, dxdy+\int_{\R^2} |\partial_{yy} u|^2 \,dxdy=\frac{3(p-2)}{4p} \int_{\R^2}|u|^p \,dx.
$$
This completes the proof.
\end{proof} 
%{\color{red}
%\begin{rem}
%The  Pohozhaev  identity established in the above Lemma  is equivalent to $Q(u)=0,$ where $Q(u)$ is the Pohozhaev quantity introduced in \eqref{rem68}.
%\end{rem}
%}

\begin{cor} \label{nonexistence}
Let $p>2$, $u \in H^{1,2}(\R^2)$ and $u \neq 0$ be a solution to \eqref{equ}. Then $\omega > 0$.
\end{cor}
\begin{proof}
Since $u \in H^{1,2}(\R^2)$ is a solution to \eqref{equ}, then
$$
\int_{\R^2}|\partial_x u|^2\, dxdy+\int_{\R^2} |\partial_{yy} u|^2 \,dxdy+\omega\int_{\R^2} |u|^2 \, dxdy=\int_{\R^2} |u|^p \, dxdy.
$$
On the other hand, by Lemma \ref{ph}, then
$$
\int_{\R^2} |\partial_x u|^2 \, dxdy + \int_{\R^2} |\partial_{yy} u|^2 \, dxdy =\frac{3(p-2)}{4p}\int_{\R^2} |u|^p \, dxdy.
$$
As a consequence, we see that
$$
\omega\int_{\R^2} |u|^2 \, dxdy=\frac{p+6}{4p }\int_{\R^2} |u|^p \, dxdy.
$$
Then $\omega>0$ and the proof is completed.
\end{proof}

\section{Local well-posedness of solutions} \label{localw}

In this section, we shall discuss the local well-posedness of solutions to \eqref{equt} and present the proof of Theorem \ref{lw}.
%and we shall start with showing the local well-posedness of solutions in $L^2(\R^2)$ for $2<p<\frac{14}{3}$. 
The Strichartz estimates for the Cauchy problem \eqref{equt} are given by
\begin{equation*}%\label{eq.NLS1210}
 \left\|e^{\textnormal{i}t (-\partial_{xx}+\partial_{yyyy})} f\right\|_{L^q_{(0,T)} L^r(\mathbb{R}^{2})} \lesssim \|f\|_{L^2(\mathbb{R}^{2})},
\end{equation*}
\begin{equation}\label{eq.NLS1211}
\left\|\int_0^t e^{\textnormal{i}(t-\tau) (-\partial_{xx}+\partial_{yyyy})} F(\tau) \, d \tau \right\|_{L^r_{(0,T)} L^q(\mathbb{R}^{2})} \lesssim \|F \|_{L^{\widetilde{r}^\prime}_{(0,T)} L^{\widetilde{q}^\prime}(\mathbb{R}^{2})} ,
\end{equation}
where $(r, q)$ and $(\widetilde{r},\widetilde{q})$ are admissible pairs. Here we say that $(r, q)$ is an admissible pair if there holds that
\begin{equation}\label{eq.NLs1212}
\frac{1}{r} + \frac{3}{4q} = \frac{3}{8}, \quad  2 \leq q \leq \infty, \quad \frac{8}{3} \leq r \leq \infty.
\end{equation}

First we are going to construct weak solutions to \eqref{equt} by applying the abstract theory of Okazawa, Suzuki, and Yokota \cite[Theorem 2.2]{OSY12}. 

\begin{lem}\label{lem:constructsol}
Let $p>2$. Then, for any $R>0$, there exists $T=T_R>0$ such that \eqref{equt} has a unique weak solution 
$$
\psi\in L^{\infty}([0,T); H^{1,2}(\mathbb{R}^2))\cap W^{1,\infty}([0,T); H^{-1,-2}(\mathbb{R}^2))
$$ 
provided initial datum $\|\psi_0\|_{H^{1,2}(\mathbb{R}^2)} \leq R$.  Moreover, the following property holds true,
$$
M(\psi(t))=M(\psi_0), \quad E(\psi(t)) \leq E(\psi_0) \quad \forall \,\, t \in [0, T).
$$
%\begin{align*}
%  &\|\psi(t)\|_{L^2}
%  =\|\psi_0\|_{L^2}, \ \ \ E(\psi(t))\le E(\psi_0)
%\end{align*}
%for all $t\in[0,T_M]$.
\end{lem}
\begin{proof}
%We will apply \cite[Theorem~2.2]{OSY12} with
Define
$$
S:= - \partial_{xx}+\partial_{yyyy}, \quad g(\psi):=-|\psi|^{p-2}\psi,
$$
$$
X:=L^2(\mathbb{R}^2),\quad X_S:=H^{1,2}(\mathbb{R}^2),\quad X_S^*:=H^{-1,-2}(\mathbb{R}^2).
$$
It is simple to check that $S$ is a nonnegative self-adjoint operator in $L^2(\mathbb{R}^2)$.  Note that $H^{1,2}(\mathbb{R}^2)=D((1+S)^{1/2})$. We now rewrite \eqref{equt} as
$$
\textnormal{i}\partial_t \psi =S\psi +g(\psi). 
$$
At this point, to derive the desired conclusions, it remains to verify that the conditions given as follows hold true.
\begin{itemize}
\item [$\textbf{(G1)}$] There exists $G\in C^1(X_S,\mathbb{R})$ such that $G'=g$.
\item [$\textbf{(G2)}$] For any $R>0$, there exists $C(R)>0$ such that
$$
\|g(u)-g(v)\|_{X_S^*}  \le C(R)\|u-v\|_{X_S}\quad  \forall \,\, u,v \,\, \in X_S, \|u\|_{X_S}, \|v\|_{X_S}\le R. 
$$
\item [$\textbf{(G3)}$] For any $R, \delta>0$, there exists $C_\delta(R)>0$ such that
$$
|G(u)-G(v)| \le \delta +C_{\delta}(R)\|u-v\|_X\quad \forall \,\, u,v\in X_S, \|u\|_{X_S}, \|v\|_{X_S}\le R.
$$
\item [$\textbf{(G4)}$] $\langle g(u),\textnormal{i} u\rangle_{X_S^*,X_S} =0$ for any $u\in X_S$.
\item [$\textbf{(G5)}$] Given a bounded open interval $I\subset\mathbb{R}$ and $\{w_n\}$ be any bounded sequence in $L^\infty(I, X_S)$ such that
$$
\left\{\begin{aligned}
w_n(t)&\to w(t)\ (n\to\infty)
&&\text{weakly in $X_S$ a.a. $t\in I$},\\
g(w_n)&\to f\ (n\to\infty)
&&\text{weakly$^*$ in $L^\infty(I,X_S^*)$}.
\end{aligned}
\right.
$$
Then
$$
\int_I\langle f(t),\textnormal{i}  w(t)\rangle_{X_S^*,X_S}\,dt =\lim_{n\to\infty}\int_I\langle g(w_n(t)),\textnormal{i} w_n(t)\rangle_{X_S^*,X_S}\,dt. 
$$
\end{itemize}
%Now we check \textbf{(G1)}--\textbf{(G5)}.
The condition \textbf{(G1)} can be easily verified as
$$
G(u)=-\frac{1}{p}\|u\|_{L^{p}}^{p},\quad u \in H^{1,2}(\mathbb{R}^2).
$$
%by standard inequalities and the embedding $H^{1,2}(\mathbb{R}^2)\hookrightarrow L^q(\mathbb{R}^2)$ $(q\ge 2)$. 
Similarly, the condition \textbf{(G2)} can be also verified. The condition \textbf{(G3)} follows from the following estimate,
\begin{align*}
 |G(u)-G(v)|
&\lesssim (\|u\|_{L^2}+\|v\|_{L^2})\bigl|\|u\|_{L^2}-\|v\|_{L^2}\bigr| +\int_{\R^2}(|u|^{p-1}+|v|^{p-1})|u-v|\,dx
\\&\lesssim(\|u\|_{L^2}+\|v\|_{L^2}+\|u\|_{L^{2(p-1)}}^{p-1}+\|v\|_{L^{2(p-1)}}^{p-1})\|u-v\|_{L^2}
\\&\lesssim (M+M^{p-1})\|u-v\|_{L^2}
\end{align*}
for $u,v\in H^{1,2}(\mathbb{R}^2)$, $\|u\|_{H^{1,2}} \leq R$ and $\|v\|_{H^{1,2}}\le R$. The condition \textbf{(G4)} is clearly from the definition of $g$. Finally, we shall check the condition \textbf{(G5)}. In view of \cite[Lemma~5.3]{OSY12}, it is enough to show that if $\{u_n\} \subset H^{1,2}(\R^2)$ satisfies that
$$
\left\{\begin{aligned}
u_n&\to u\ (n\to\infty)
&&\text{weakly in $H^{1,2}(\mathbb{R}^2)$},
\\g(u_n)&\to f\ (n\to\infty)
  &&\text{weakly in $H^{-1,-2}(\mathbb{R}^2)$},
\end{aligned}\right. 
$$
then $f=g(u)$. For this, we shall follow the arguments as in the proof of \cite[Theorem 1.1]{S15}. Let $\varphi\in C_{0}^\infty(\mathbb{R}^2)$. From the weak convergence of $\{u_n\}$ in $H^{1,2}(\mathbb{R}^2)$ and the compactness $H_{loc}^{1,2}(\mathbb{R}^2)\hookrightarrow L_{{loc}}^{p}(\mathbb{R}^2)$, then $u_n\to u$ in $L_{{loc}}^{p}(\mathbb{R}^2)$ as $n \to \infty$. Therefore, we have that
\begin{align*}
 |\langle g(u_n)-g(u),\varphi\rangle|
\lesssim |\langle u_n-u, \varphi\rangle|+\|\varphi\|_{L^{p}}(\|u_n\|_{L^{p}}^{p-2}+\|u\|_{L^{p}}^{p-2})\|u_n-u\|_{L^{p}}=o_n(1).
 \end{align*}
It then means that $g(u_n)\to g(u)$ in $\mathcal{D}'(\mathbb{R}^2)$ as $n \to \infty$. On the other hand, we know that $g(u_n)\to f$ in $H^{-1,-2}(\mathbb{R}^2)$ as $n \to \infty$. As a consequence, we obtain that $f=g(u)$. Thus the condition \textbf{(G5)} is verified. Hence \cite[Theorem~2.2]{OSY12} readily implies the conclusion and the proof is completed.
%We have just finished the verification of \textbf{(G1)}--\textbf{(G5)}. Therefore, \cite[Theorem~2.2]{OSY12} implies the conclusion.
\end{proof}

%Next, we establish the uniqueness. We apply the same argument in the proof of \cite[Lemma 3.1]{S15}. 

\begin{lem}\label{lem:uniquenesssol}
Let $p>2$ and $\psi_0 \in H^{1,2}(\mathbb{R}^2)$. If $\psi_1, \psi_2\in L^\infty([0, T); H^{1,2}(\mathbb{R}^2))$ are two weak solutions \eqref{equt} with $\psi_1(0)=\psi_2(0)=u_0$, then $\psi_1=\psi_2$.
\end{lem}
\begin{proof} 
Observe that
\begin{equation}
\psi_j(t):= e^{\textnormal{i} \left(-\partial_{xx} +\partial_{yyyy}\right) t}\psi_0  -\textnormal{i} \int_0^t e^{\textnormal{i}\left(-\partial_{xx} +\partial_{yyyy}\right)(t-\tau)} |\psi_j(\tau)|^{p-2} \psi_j(\tau) \,d \tau, \quad j=1,2.
\end{equation}
Using the Strichartz estimate \eqref{eq.NLS1211} by properly choosing $r>2$ implies that
$$ 
\left\|\psi_1-\psi_2\right\|_{L^r_{(0,T)} L^{p}(\mathbb{R}^{2})} \lesssim  \||\psi_1|^{p-2}\psi_1-|\psi_2|^{p-2}\psi_2\|_{L^{r'}_{(0, T)} L^{p'}(\mathbb{R}^{2})}.
$$
Note that
$$
\||\psi_1|^{p-2}\psi_1-|\psi_2|^{p-2}\psi_2\|_{L^{p'}(\mathbb{R}^{2})} \lesssim \left(\|\psi_1\|_{L^p(\mathbb{R}^{2})}^{p-2}+\|\psi_2\|_{L^p(\mathbb{R}^{2})}^{p-2}\right)\|\psi_1-\psi_2\|_{L^p(\mathbb{R}^{2})} \lesssim \|\psi_1-\psi_2\|_{L^p(\mathbb{R}^{2})}.
$$
It then follows that
$$
\left\|\psi_1-\psi_2\right\|_{L^r_{(0,T)} L^{p}(\mathbb{R}^{2})} \leq CT^{1-\frac 2 r} \left\|\psi_1-\psi_2\right\|_{L^r_{(0,T)} L^{p}(\mathbb{R}^{2})}.
$$
This then infers the desired conclusion and the proof is completed.
%To prove the uniqueness of the solutions, we shall apply the same arguments as in the proof of \cite[Lemma 3.1]{S15}.  Taking into account the Strichartz estimates \eqref{eq.NLS1210} and \eqref{eq.NLS1211} as in the proof of Lemma \ref{lwpl2}, we can get that
%$$
%\|\psi_1-\psi_2\|_{L^q_{(0,T)} L^r(\mathbb{R}^{2})}  \lesssim T^{\alpha} \|\psi_1-\psi_2\|_{L^q_{(0,T)} L^r(\mathbb{R}^{2})}.
%$$
\end{proof}

\begin{proof}[Proof of Theorem \ref{lw}]
The existence and uniqueness of solutions and the conservation laws follow directly from Lemmas \ref{lem:constructsol} and \ref{lem:uniquenesssol} together with \cite[Theorem~2.3]{OSY12}. The remaining points can be proved by standard ways.
%When $p \in (2,3]$, the conclusion follows from the application of Lemmas \ref{lem:constructsol} and \ref{lem:uniquenesssol} together with \cite[Theorem~2.3]{OSY12}. While $p>3$, the conclusion follows from Lemma \ref{lwpl38}. 
This completes the proof.
\end{proof}

\section{Existence and quantitative properties of ground states} \label{properties}

In the section, we shall discuss the existence, axial symmetry, decay and orbital stability/instability of solutions to \eqref{equ} and present the proofs of Theorems \ref{groundstate}-\ref{instability}.

\begin{proof} [Proof of Theorem \ref{groundstate}]
If $u \in N$, then $I_{\omega}(u)=0$. It follows that
\begin{align} \label{bdd}
J_{\omega}(u)=J_{\omega}(u)-\frac 1p I_{\omega}(u)=\frac{p-2}{2p} \left(\int_{\R^2} |\partial_x u|^2 \,dxdy +\int_{\R^2} |\partial_{yy} u|^2 \,dxdy + \omega \int_{\R^2} |u|^2 \,dxdy \right).
\end{align}
Moreover, by \eqref{gn}, then 
\begin{align*}
&\int_{\R^2} |\partial_x u|^2 \,dxdy +\int_{\R^2} |\partial_{yy} u|^2 \,dxdy + \omega \int_{\R^2} |u|^2 \,dxdy \leq C_{opt} \|\partial_x u\|_2^{\frac{p-2}{2}} \|\partial_{yy} u\|_2^{\frac{p-2}{4}} \|u\|_2^{\frac{p+6}{4}} \\
& \leq C_{opt} \left(\int_{\R^2} |\partial_x u|^2 \,dxdy +\int_{\R^2} |\partial_{yy} u|^2 \,dxdy + \omega \int_{\R^2} |u|^2 \,dxdy\right)^{\frac p 2}.
\end{align*}
Due to $p>2$, we then have that
$$
\int_{\R^2} |\partial_x u|^2 \,dxdy +\int_{\R^2} |\partial_{yy} u|^2 \,dxdy + \omega \int_{\R^2} |u|^2 \,dxdy \geq C_{opt}^{\frac{2}{2-p}}.
$$
It then yields from \eqref{bdd} that $m_{\omega}>0$. Observe that $N$ is a natural constraint. As a consequence, there exists a Palais-Smale sequence $\{u_n\} \subset N$ for $J_{\omega}$ at the level $m_{\omega}$, i.e. $J'_{\omega}(u_n)=o_n(1)$ and $J_{\omega}(u_n)=m+o_n(1)$. Clearly, by using \eqref{bdd}, then $\{u_n\}$ is bounded in $H^{1,2}(\R^2)$. From Lemma \ref{ccl}, then there exist a sequence $\{z_n\}\subset \R^2$ and a nontrivial $u \in H^{1,2}(\R^2)$ such that $u_n(\cdot+z_n) \wto u$ in $H^{1,2}(\R^2)$ as $n \to \infty$. In addition, we see that $u$ is a solution to \eqref{equ}. It then follows that $I_{\omega}(u)=0$ and $J_{\omega}(u) \geq m_{\omega}$. 
%Define $w_n:=u_n(\cdot +z_n)$, then
%$$
%E(w_n-u) +E(u)=E(w_n) +o_n(1), \quad I(w_n-u) +I(u)=I(w_n) +o_n(1).
%$$
%If $I(u)>0$, then $I(w_n-u)<0$ for any $n \in \N^+$ large enough. It is simple to see that there exists $0<t_n<1$ such that %$I(t_n(w_n-u))=0$ for any $n \in \N^+$ large enough. Observe that
%\begin{align*}
%m \leq E(t_n(w_n-u))&=\frac{(p-2)t_n}{2p} \left(\int_{\R^2} |\partial_x u|^2 \,dxdy +\int_{\R^2} |\partial_{yy} u|^2 \,dxdy + \omega \int_{\R^2} |u|^2 \,dxdy \right)\\
%&<\frac{p-2}{2p} \left(\int_{\R^2} |\partial_x u|^2 \,dxdy +\int_{\R^2} |\partial_{yy} u|^2 \,dxdy + \omega \int_{\R^2} |u|^2 \,dxdy \right) \\
%&=E(w_n-u)-\frac 1p I(w_n-u).
%\end{align*}
By Fatou's Lemma, there holds that
\begin{align*}
m_{\omega}+o_n(1)=J_{\omega}(u_n)&=J_{\omega}(u_n) - \frac 12 I_{\omega}(u_n) \\
&=\frac{p-2}{2p} \int_{\R^2} |u_n|^p \,dxdy \\
&\geq \frac{p-2}{2p} \int_{\R^2} |u|^p\,dxdy+o_n(1) =J_{\omega}(u)-\frac 1p I_{\omega}(u)=J_{\omega}(u).
\end{align*}
Then we conclude that $m_{\omega}=J_{\omega}(u)$ and $u \in H^{1,2}(\R^2)$ is a ground state to \eqref{equ}. The proof is completed.
\end{proof}

\begin{proof}[Proof of Theorem \ref{symmetry}]
First we prove axial symmetry of ground states to \eqref{equ} with respect to $x$-axis. Here we are going to make use of arguments from \cite{L}. Let $u \in H^{1,2}(\R^2)$ be a ground state to \eqref{equ}. Let $x_0 \in \R$ be such that 
$$
\int_{x_0}^{+\infty} \int_{\R}|u|^p\,dxdy=\int_{-\infty}^{x_0} \int_{\R} |u|^p \,dxdy.
$$
For simplicity, we shall assume that $x_0=0$ by making a translation, i.e.
\begin{align} \label{lp}
\int_{\R^2_+}|u|^p\,dxdy=\int_{\R^2_-} |u|^p \,dxdy,
\end{align}
where $\R^2_+:=(0, +\infty) \times \R$ and $\R^2_-:=\R^2 \backslash \R^2_+$. . Define
\begin{align*}
\widetilde{u}(x, y):=
\left\{
\begin{aligned}
&u(x, y), &\mbox{if} \,\,\, x>0,\\
& u(-x, y), &\mbox{if} \,\,\, x \leq 0.
\end{aligned}
\right.
\end{align*}
Then there holds that
\begin{align} \label{pne} 
\begin{split}
&\frac 12 \int_{\R^2_+} |\partial_x u|^2 \, dxdy+ \frac 12 \int_{\R^2_+} |\partial_{yy} u|^2 \, dxdy + \frac 12 \int_{\R^2_+} |\partial_x u|^2 \, dxdy - \frac 1 p\int_{\R^2_+} |u|^p \, dxdy \\
&=\frac 12 \int_{\R^2_-} |\partial_x u|^2 \, dxdy+ \frac 12 \int_{\R^2_-} |\partial_{yy} u|^2 \, dxdy + \frac 12 \int_{\R^2_-} |\partial_x u|^2 \, dxdy - \frac 1 p\int_{\R^2_-} |u|^p \, dxdy,
\end{split}
\end{align}
Indeed, to prove \eqref{pne}, we shall argue by contradiction. Let us denote $A$ and $B$ by the left hand side and the right side hand of \eqref{pne}, respectively. Then we may assume that $A<B$.
Observe that
$$%\begin{align} \label{gx1}
%\int_{\R^2} |\partial_x \widetilde{u}|^2 \, dxdy=\int_0^{+\infty} \int_{\R} |\partial_x \widetilde{u}|^2 \, dxdy+\int_{-\infty}^{0} \int_{\R} |\partial_x \widetilde{u}|^2 \, dxdy=2\int_0^{+\infty} \int_{\R} |\partial_x u|^2 \, dxdy,
\int_{\R^2_+} |\partial_x \widetilde{u}|^2 \, dxdy=\int_{\R^2_-} |\partial_x \widetilde{u}|^2 \, dxdy=\int_{\R^2_+} |\partial_x u|^2 \, dxdy,
$$%\end{align}
$$%\begin{align} \label{gy2}
\int_{\R^2_+} |\partial_{yy} \widetilde{u}|^2 \, dxdy=\int_{\R^2_-} |\partial_{yy} \widetilde{u}|^2 \, dxdy=\int_{\R^2_+} |\partial_{yy} u|^2 \, dxdy, 
$$%\end{align}
$$%\begin{align} \label{l2}
\int_{\R^2_+} |\widetilde{u}|^2 \, dxdy=\int_{\R^2_-} | \widetilde{u}|^2 \, dxdy=\int_{\R^2_+} |u|^2 \, dxdy.
$$%\end{align}
Applying \eqref{lp} and the assumption that $A<B$, we then get that $J_{\omega}(\widetilde{u})<J_{\omega}(u)$ and $I_{\omega}(\widetilde{u})<I_{\omega}(u)=0$. Hence there exists $0<t_{\widetilde{u}}<1$ such that $I_{\omega}(t_{\widetilde{u}}\widetilde{u})=0$. Therefore, there holds that 
\begin{align*}
m_{\omega} \leq J_{\omega}(t_{\widetilde{u}}\widetilde{u})<J_{\omega}(t_{\widetilde{u}}u) < J_{\omega}(u)=m_{\omega}.
\end{align*}
Then we reach a contradiction. This implies that $A=B$ and \eqref{pne} holds true. At this point, using  \eqref{pne}, we can conclude that $J_{\omega}(\widetilde{u})=m_{\omega}$ and $I_{\omega}(\widetilde{u})=0$. It then follows $\widetilde{u}$ is also a ground state to \eqref{equ}. Invoking unique continuation principle \cite[Theorem A.1]{dS}, we then have that $u=\widetilde{u}$ and the desired conclusion follows.

Next we demonstrate axial symmetry of ground states to \eqref{equ} with respect to $x$-axis and $y$-axis for $p \in \mathbb{N}$, which is inspired by \cite{BL, LS}. Define Fourier rearrangement of $u$ with respect to the vector {\bf{e}}  with $|{\bf{e}}|=1$ by
$$
u^{\# {\bf{e}}}:=\mathcal{F}^{-1} \left\{(\mathcal{F} u)^{\ast {\bf{e}}}\right\},
$$
where $f^{\ast {\bf{e}}}$ denotes the Steiner rearrangement of $f$ with respect to the vector ${\bf {e}}$ in $\R^2$, $\mathcal{F} f$ denotes the Fourier transform of $f$ given by
$$
(\mathcal{F} f)(\xi):=\int_{\R^2} f(z) e^{-2\pi \textnormal{i} z \cdot \xi}\, dz.
$$
And $\mathcal{F}^{-1} f$ denotes the inverse Fourier transform of $f$. It follows from the definition of Fourier rearrangement that $u^{\ast {\bf{e}}}$ is symmetric with respect to ${\bf{e}}$. Reasoning as the proof of \cite[Lemma A.1]{BL}, we can get that
\begin{align} \label{rea1}
\left\|\partial_x u^{\# {\bf{e}}}\right\|_2 \leq \left\|\partial_x u\right\|_2, \quad \left\|\partial_{yy} u^{\# {\bf{e}}}\right\|_2 \leq \left\|\partial_{yy} u\right\|_2, %\quad \left\|u\right\|_{2m} \leq \left\|u^{\# {\bf{e}}}\right\|_{2m}, \quad m \in \mathbb{N}.
\end{align}
and
\begin{align} \label{rea2}
%\left\|\partial_x u^{\# {\bf{e_2}}}\right\|_2 \leq \left\|\partial_x u\right\|_2, \quad \left\|\partial_{yy} u^{\# {\bf{e_2}}}\right\|_2 \leq \left\|\partial_{yy} u\right\|_2, \quad 
\left\|u\right\|_{2m} \leq \left\|u^{\# {\bf{e}}}\right\|_{2m}, \quad m \in \mathbb{N}.
\end{align}
%where ${\bf{e_1}}:=(0,1)$ and ${\bf{e_2}}:=(1,0)$.
Following the proof of \cite[Theorem 1]{LS}, we can derive that the inequalities in \eqref{rea1} and \eqref{rea2} occur if and only if $u$ equals to $u^{\# \bf{e}}$ up to a constant phase and translation, i.e.
$$
u(z)=e^{\textnormal{i} \alpha} u^{\# \bf{e}}(z-z_0), \quad \alpha \in \R,  z_0 \in \R^2.
$$
Let $u \in H^{1,2}(\R^2)$ be a ground state to \eqref{equ}. Define 
$$
w:=\left(u^{\# {\bf{e_1}}}\right)^{\# {\bf {e_2}}}.
$$
Using \eqref{rea1} and \eqref{rea2}, we then obtain that $J_{\omega}(w) \leq J_{\omega}(u)$ and $I_{\omega}(w) \leq I_{\omega}(u)=0$. Then we see that there exists $0<t_w \leq 1$ such that $I_{\omega}(t_ww)=0$. Therefore, from \eqref{rea1} and \eqref{rea2}, we have that
\begin{align*}
m_{\omega} \leq J_{\omega}(t_w w) &= J_{\omega}(t_w w)-\frac 1p I_{\omega}(t_w w)\\
&=\frac{t_w(p-2)}{2p} \left(\int_{\R^2} |\partial_x w|^2 \,dxdy +\int_{\R^2} |\partial_{yy} w|^2 \,dxdy + \omega \int_{\R^2} |w|^2 \,dxdy \right) \\
& \leq \frac{p-2}{2p} \left(\int_{\R^2} |\partial_x u|^2 \,dxdy +\int_{\R^2} |\partial_{yy} u|^2 \,dxdy + \omega \int_{\R^2} |u|^2 \,dxdy \right)\\
&=J_{\omega}(u) -\frac 1p I_{\omega}(u)=J_{\omega}(u)=m_{\omega},
\end{align*}
from which we derive that $m_{\omega}=J_{\omega}(t_w w)$ and  $t_{w}=1$. Moreover, we have that
$$
\left\|\partial_x u\right\|_2 = \left\|\partial_x w\right\|_2, \quad \left\|\partial_{yy} w\right\|_2 = \left\|\partial_{yy} u\right\|_2, \quad \left\|u\right\|_{p} = \left\|w\right\|_p.
$$ 
It then follows that
$$
u(z)=e^{\textnormal{i} \alpha} \left(u^{\# {\bf{e_1}}}\right)^{\# {\bf {e_2}}}(z-z_0), \quad \alpha \in \R, z_0 \in \R^2.
$$
This completes the proof.
\end{proof}

\begin{proof}[Proof of Theorem \ref{decay}]
Observe that $u \in H^{1,2}(\R^2)$ is a solution to \eqref{equ}, then $v:=\omega^{\frac 1p} u(\omega^{\frac 12} \cdot, \omega^{\frac 14} \cdot) \in H^{1,2}(\R^2)$ is a solution to \eqref{equ} with $\omega=1$. For simplicity, we shall assume that $\omega=1$. Let $\mathcal{K}$ denote the fundamental solution to the equation
$$
-\partial_{xx} u+ \partial_{yyyy} u + u=0 \quad \mbox{in} \,\, \R^2.
$$
It is characterized by
$$
\mathcal{K}(x ,y)=\int_0^{+\infty} e^{-t} \mathcal{H}(x, y, t) \, dt,
$$
where
$$
\mathcal{H}(x ,y, t):=\int_{\R^2} e^{-2\pi (x, y) \cdot(\xi_1, \xi_2)-t (|\xi_1|^2 +|\xi_2|^4)} \,d\xi_1d\xi_2.
$$
Indeed, for any $h \in \mathcal{S}$, we see that
\begin{align*}
\langle \mathcal{K}, h \rangle&=\int_{\R^2} \int_0^{+\infty} \int_{\R^2} e^{-2\pi \textnormal{i} (x, y) \cdot(\xi_1, \xi_2)-t (1+|\xi_1|^2 +|\xi_2|^4)} h(x ,y)\,d\xi_1d\xi_2 \, dt dxdy \\
&=\int_{\R^2}  \int_{\R^2} e^{-2\pi \textnormal{i} (x, y) \cdot(\xi_1, \xi_2)} h(x ,y) \int_0^{+\infty} e^{-t (1+|\xi_1|^2 +|\xi_2|^4)} \,dt \,d\xi_1d\xi_2 \, dxdy \\
&=\int_{\R^2} \frac{1}{1+|\xi_1|^2 +|\xi_2|^4} \int_{\R^2} e^{-2\pi \textnormal{i} (x, y) \cdot(\xi_1, \xi_2)} h(x ,y) \,dxdy d\xi_1d\xi_2\\
&=\left\langle \frac{1}{1+|\xi_1|^2 +|\xi_2|^4}, \mathcal{F} h \right\rangle.
\end{align*}
Observe that
\begin{align*}
\mathcal{H}(x,y, t)&=\int_{\R^2} e^{-2\pi \textnormal{i}(x, y) \cdot(\xi_1, \xi_2)-t(|\xi_1|^2 +|\xi_2|^4)} \,d\xi_1d\xi_2= \int_{\R}  e^{-2\pi \textnormal{i}x\xi_1-t|\xi_1|^2} \,d\xi_1 \int_{\R} e^{-2\pi \textnormal{i} y \xi_2-t|\xi_2|^4} \,d\xi_2.
\end{align*}
Moreover, there holds that
\begin{align} \label{h1}
\mathcal{H}_1(x ,t):=\int_{\R}  e^{-2\pi \textnormal{i}x\xi_1-t|\xi_1|^2} \,d\xi_1= \sqrt{\frac{\pi}{t}} e^{-\frac{\pi^2 |x|^2}{t}},
\end{align}
\begin{align} \label{h2}
\mathcal{H}_2(y,t):=\int_{\R} e^{-2\pi \textnormal{i} y \xi_2-t|\xi_2|^4} \,d\xi_2=t^{-\frac 14}\int_{\R} e^{-2\pi \textnormal{i} t^{-\frac 14} y\xi_2-|\xi_2|^4} \,d\xi_2=t^{-\frac 14} \mathcal{H}_2\left(t^{-\frac 14} y, 1\right).
\end{align}
It follows from \cite{Bo} that there exist $c_0, c_1, c_2>0$ such that, for $y>0$ large enough,
$$
\mathcal{H}_2(y ,1)=\int_{\R} e^{-2\pi \textnormal{i} y \xi_2-|\xi_2|^4} \,d\xi_2 \sim  c_0 y^{-\frac 13}  e^{-c_1 y^{\frac 43}} \cos\left(c_2 y^{\frac 43}-\frac{\pi}{6}\right).
$$
%\begin{align*}
%\frac{d^3}{dy^3} F(y)=8\pi^3\textnormal{i}\int_{\R} e^{-2\pi \textnormal{i} y \xi_2-|\xi_2|^4} \xi_2^3 \,d\xi_2&=-2\pi^3\textnormal{i}\int_{\R} e^{-2\pi \textnormal{i} y \xi_2} \,d (e^{-|\xi_2|^4})\\
%&=-4\pi^4 y\int_{\R} e^{-2\pi \textnormal{i} y \xi_2-|\xi_2|^4}\,d\xi_2=-4\pi^4 y F(y).
%\end{align*}
%This leads to 
%$$
%F(y)=e^{-C y^{\frac 4 3}}.
%$$
In light of \eqref{h2}, we then find that, for any $y \in \R$ and $t>0$,
\begin{align} \label{eh2}
\left|\mathcal{H}_2(y, t) \right|\leq  c_3 \min \left\{t^{-\frac 14}, t^{-\frac 16} |y|^{-\frac 13}e^{-c_1 t^{-\frac 13} |y|^{\frac 43}}\right\}.
\end{align}
As a consequence, by \eqref{h1} and \eqref{eh2}, we are able to compute that, for any $q \geq 1$,
\begin{align*}
\left\|\mathcal{H}\right\|_q^q =\left\|\mathcal{H}_1\mathcal{H}_2\right\|_q^q &\leq \pi^{\frac q 2} \int_{\R}\int_{|y| \leq y_t} t^{-\frac{3q}{4}} e^{-\frac{q \pi |x|^2}{t}} \,dxdy + \pi^{\frac q 2}  \int_{\R}\int_{|y| > y_t} t^{-\frac{2q}{3}} e^{-\frac{q \pi |x|^2}{t}} |y|^{-\frac q 3}e^{-c_1^q t^{-\frac q3} |y|^{\frac {4q}{3}}}\,dxdy \\
& \leq c_4 t^{-\frac{3}{4} \left(q-1\right)},
\end{align*}
where $y_t>0$ satisfies the identity
$$
\frac {1}{12} \ln t=\frac 13 \ln y_t + c_1 t^{-\frac 13} y_t^{\frac 4 3}.
$$
Therefore, by Minkowski's inequality, we conclude that
\begin{align*}
\left\|\mathcal{K}\right\|_q \leq \int_0^{+\infty} e^{-t} \left\|\mathcal{H}\right\|_q \, dt \leq c_4^{\frac 1q} \int_0^{+\infty} e^{-t} t^{-\frac{3}{4} + \frac{3}{4q}}\, dt<+\infty.
\end{align*}
It then implies that $\mathcal{K} \in L^q(\R^2)$ for any $q \geq 1$. Note that we assumed that $u \in H^{1,2}(\R^2)$ is a solution to \eqref{equ} with $\omega=1$, then
$$
u=\mathcal{K} \ast (|u|^{p-2} u) =\int_{\R^2} \mathcal{K}(z-\xi) |u|^{p-2} u(\xi) \, d\xi, \quad z=(x, y), \,\,\, \xi=(\xi_1, \xi_2) \in \R^2.
$$
Using H\"older's inequality and the embedding theorem in $H^{1,2}(\R^2)$, we then have that $u \in L^{\infty}(\R^2)$. Further, by standard bootstrap arguments, we then derive that $u \in C^{\infty}(\R^2)$. Combining \eqref{h1} and \eqref{eh2}, we have that, for any $(x, y) \in \R^2$ satisfying $|x| +|y|>0$ large,
\begin{align*}
|\mathcal{K}(x ,y)| &\leq C |y|^{-\frac 13}\int_{0}^{+\infty} t^{-\frac 23}e^{-t-\frac{\pi^2 |x|^2}{t}-c_1 t^{-\frac 13} |y|^{\frac 43}} \,dt \\
& \leq C |y|^{-\frac 13} \int_{0}^{1}  t^{-\frac 23}e^{-\left(t+t^{-\frac 13}\left(\pi^2 |x|^2+c_1 |y|^{\frac 43}\right)\right) }\,dt + C |y|^{-\frac 13} \int_{1}^{+\infty}  t^{-\frac 23}e^{-\left(t+t^{-1}\left(\pi^2 |x|^2+c_1 |y|^{\frac 43}\right)\right)} \,dt \\
& \leq C |y|^{-\frac 13} \int_{0}^{1}  t^{-\frac 23} e^{-\left(1+\pi^2 |x|^2+c_1 |y|^{\frac 43}\right)} \, dt + C |y|^{-\frac 13}\int_{1}^{+\infty}  t^{-\frac 23} e^{-\frac t 2}e^{-\left(2\pi^2 |x|^2+2c_1 |y|^{\frac 43}\right)^{\frac 12}} \, dt \\
& \leq C |y|^{-\frac 13} e^{-\left(2\pi^2 |x|^2+2c_1 |y|^{\frac 43}\right)^{\frac 12}},
\end{align*}
where we used the fact that the function $t \mapsto t+at^{-\frac 13}$ is decreasing on $(0, 1)$ for $a>0$ large and the inequality
$$
\frac t2 + t^{-1}b \geq \sqrt{2b}, \quad b>0.
$$
Then we get that there exist $C_1, C_2>0$ such that, for any $x, y \in \R$,
$$
|\mathcal{K}(x,y)| \leq C_1 |y|^{-\frac 13} e^{-C_2\left(|x|+|y|^{\frac 23}\right)}.
$$
At this point, reasoning as the proof of \cite[Corollary 3.1.4]{BoLi}, we then have the desired result. This completes the proof.
\end{proof}

\begin{proof}[Proof of Theorem \ref{stability}]
Let $u_{\omega} \in H^{1,2}(\R^2)$ be a ground state to \eqref{equ} at the energy level $m_{\omega}>0$. Then there holds that
$$
m_{\omega}=\frac 12 \int_{\R^2} |\partial_x u_{\omega}|^2 \,dxdy +\frac 12 \int_{\R^2} |\partial_{yy} u_{\omega}|^2 \,dxdy + \frac{\omega}{2} \int_{\R^2} |u_{\omega}|^2 \,dxdy -\frac 1p \int_{\R^2} |u_{\omega}|^p \, dxdy,
$$
$$
\int_{\R^2} |\partial_x u_{\omega}|^2 \,dxdy +\int_{\R^2} |\partial_{yy} u_{\omega}|^2 \,dxdy + \omega \int_{\R^2} |u_{\omega}|^2 \,dxdy =\int_{\R^2} |u_{\omega}|^p \, dxdy
$$
and
$$
\int_{\R^2} |\partial_x u_{\omega}|^2 \,dxdy +\int_{\R^2} |\partial_{yy} u_{\omega}|^2 \,dxdy =\frac{3(p-2)}{4p}\int_{\R^2} |u_{\omega}|^p \, dxdy,
$$
where the last identity comes from Lemma \ref{ph}. It is not hard to obtain from the identities above that
$$
m_{\omega}=\frac{2(p-2)\omega}{p+6} \int_{\R^2} |u_{\omega}|^2 \,dx.
$$
This indicates that all ground states to \eqref{equ} admit the same $L^2$-norm. Let us now define the following minimization problem,
\begin{align} \label{min2}
m(c):=\inf_{v \in S(c)} E(v), 
\end{align}
where
$$
S(c):=\left\{ v \in H^{1,2}(\R^2) : \|v\|_2^2=c>0\right\}.
$$
Next we shall claim that $u_{\omega} \in H^{1,2}(\R^2)$ is a ground state to \eqref{equ} if and only if it is a minimizer to \eqref{min2} with $c=\|u_{\omega}\|_2^2$. To prove this,  we first define $v_t:=t^{\frac 38} v(t^{\frac 12} \cdot, t^{\frac 14} \cdot)$ for any $v \in H^{1,2}(\R^2)$ and $t>0$. Direct computations lead to $\|v_t\|_2=\|v\|$ and
\begin{align} \label{scaling1}
E(v_t)=\frac{t}{2} \int_{\R^2} |\partial_x v|^2 \, dxdy + \frac{t}{2} \int_{\R^2} |\partial_{yy} v|^2 \, dxdy -\frac{t^{\frac 38 (p-2)}}{p} \int_{\R^2} |v|^p \, dxdy.
\end{align}
Thanks to $2<p<\frac{14}{3}$, from \eqref{scaling1}, then $E(v_t)<0$ for any $t>0$ small. This clearly shows that $m(c)<0$ for any $c>0$. On the other hand, from \eqref{gn}, then
\begin{align*}
E(v) \geq \frac{1}{2} \int_{\R^2} |\partial_x v|^2 \, dxdy + \frac{1}{2} \int_{\R^2} |\partial_{yy} v|^2 \, dxdy -\frac{C_{opt}}{p} \left(\int_{\R^2} |\partial_x v|^2 \, dxdy +\int_{\R^2} |\partial_{yy} v|^2 \, dxdy \right)^{\frac{3p-6}{8}} c^{\frac{p+6}{8}}.
\end{align*}
It gives that $m(c)>-\infty$, because of $2<p<\frac{14}{3}$. Consequently, we have that $-\infty<m(c)<0$ for any $c>0$. In addition, by scaling technique, it is standard to obtain that
$$
m(c_1+c_2)<m(c_1)+m(c_2),
$$
where $c_1, c_2>0$. In the spirit of the Lions concentration compactness principle in \cite{Li1, Li2}, with the help of Lemma \ref{ccl}, we then derive that any minimizing sequence to \eqref{min2} is compact in $H^{1,2}(\R^2)$ up to translations for any $c>0$. In particular, there exists a minimizer to \eqref{min2} for any $c>0$. Observe that
\begin{align}\label{mp}
m_{\omega}=\inf_{u \in P} J_{\omega}(u),
\end{align}
where
$$
P:=\{ u \in H^{1,2}(\R^2) \backslash \{0\} : Q(u)=0\}.
$$
%where 
%$$
%Q(u):=\int_{\R^2} |\partial_x u|^2 \, dxdy + \int_{\R^2} |\partial_{yy} u|^2 \, dxdy -\frac{3(p-2)}{4p}\int_{\R^2} |u|^p \, dxdy.
%$$
Let $u_{\omega} \in H^{1,2}(\R^2)$ be a ground state to \eqref{equ}. Define $c:=\|u_{\omega}\|_2^2$. Let $v \in S(c)$ be a minimizer to \eqref{min1}, then $E(v)=m(c)$ and $v$ solves the equation
$$
-\partial_{xx} v+ \partial_{yyyy} v +\omega_c v=|v|^{p-2} v \quad \mbox{in} \,\, \R^2,
$$
where $\omega_c \in \R$ is Lagrange multiplier related to the constraint defined by
$$
\omega_c=\frac 1 c \left(\int_{\R^2} |u|^p \, dxdy-\int_{\R^2} |\partial_x v|^2 \, dxdy -\int_{\R^2} |\partial_{yy} v|^2 \, dxdy\right).
$$  
It then follows from Lemma \ref{ph} that $Q(v)=0$. Therefore, applying \eqref{mp}, we derive that
$$
m(c) \leq E(u_{\omega})=J_{\omega}(u_{\omega})-\frac{\omega}{2} \|u_{\omega}\|_2^2 \leq J_{\omega}(v)-\frac{\omega}{2} \|v\|_2^2=E(v)=m(c).
$$
Thus the claim follows. Arguing by contradiction and following the strategies in \cite{CL}, we have that the set of minimizers to \eqref{min1} is orbitally stable for any $c>0$. Making use of the claim, we then get the desired conclusion. This completes the proof.
\end{proof}

Next we are going to discuss orbital instability of standing waves associated to ground states to \eqref{equ} for $p>\frac{14}{3}$. For this, we shall make use of arguments from \cite{SS}. The following study is also inspired by \cite{BIK}. In the sequel, we denote by $u_{\omega}$ a ground state to \eqref{equ}. Define
$$
\mathcal{U}_{\eps}(u_{\omega}):=\left\{ u \in H^{1,2}(\R^2) : \inf_{\theta \in \R} \|u-e^{\textnormal{i}\theta}u_{\omega}\|_{H^{1,2}} <\eps\right\},
$$
$$
\psi_{\omega}:=\frac 3 8 u_{\omega} + \frac 12 x \partial_x u_{\omega} + \frac 14 y \partial_y u_{\omega}.
$$
For $u \in \mathcal{U}_{\eps}(u_{\omega})$, we define
\begin{align} \label{defau}
A(u):=-\langle e^{\textnormal{i} \alpha(u)} u, \textnormal{i} \psi_{\omega} \rangle, \quad \alpha(u):=-\tan^{-1} \frac{\langle I(u), \textnormal{i} u_{\omega}\rangle}{\langle I(u), u_{\omega} \rangle},
\end{align}
where $I$ is a natural isomorphism from $H^{1,2}(\R^2)$ to its dual space defined by $\langle I(u), v \rangle:=(u, v)$ and $(u, v)$ stands for the inner product in $H^{1,2}(\R^2)$ for $u, v \in H^{1,2}(\R^2)$.

\begin{lem} \label{propau}
For any $\eps>0$, there holds that $A: \mathcal{U}_{\eps}(u_{\omega}) \to \R$ is a $C^1$ functional such that
\begin{itemize}
\item [$(\textnormal{i})$] $A(e^{\textnormal{i} \theta} u)=A(u)$,
\item [$(\textnormal{ii})$] $A'(u_{\omega})= -\textnormal{i} \psi_{\omega}$,
\item [$(\textnormal{iii})$] $\mathcal{R} (A'(u)) \subset X$,
\item [$(\textnormal{iv})$] $\langle u, \textnormal{i} A'(u) \rangle=0$.
\end{itemize}
\end{lem}
\begin{proof}
In view of the definition of $A(u)$, we first know that $A: \mathcal{U}_{\eps}(u_{\omega}) \to \R$ is of class $C^1$. Note that, for any $\theta \in \R$,
$$
e^{\textnormal{i} \alpha(e^{\textnormal{i} \theta} u)}e^{\textnormal{i} \theta} u=e^{\textnormal{i} \alpha(u)} u.
$$ 
Then $A(e^{\textnormal{i} \theta} u)=A(u)$ for any $u \in \mathcal{U}_{\eps}(u_{\omega})$. Thus the assertion $(\textnormal{i})$ follows. For any $h \in H^{1,2}(\R^2)$, we see that
$$
A'(u) h=-\langle e^{\textnormal{i} \alpha(u)} h, \textnormal{i} \psi_{\omega} \rangle-\langle e^{\textnormal{i} \alpha(u)} u, \textnormal{i} \psi_{\omega} \rangle \langle \textnormal{i} \alpha'(u), h \rangle.
$$
It results in
\begin{align} \label{dau}
A'(u)=-\textnormal{i}e^{-\textnormal{i} \alpha(u)}\psi_{\omega}- \langle e^{\textnormal{i} \alpha(u)} u, \psi_{\omega} \rangle\alpha'(u).
\end{align}
It is obvious that $\alpha(u_{\omega})=0$ and $\langle u_{\omega}, \psi_{\omega} \rangle=0$. It then follows from \eqref{dau} that
$$
A'(u_{\omega})=-\textnormal{i}\psi_{\omega}-\langle u_{\omega}, \psi_{\omega} \rangle \alpha'(u)=-\textnormal{i}\psi_{\omega}.
$$
This proves the assertion $(\textnormal{ii})$. Observe that $\alpha'(u) \in H^{1,2}(\R^2)$ for any $u \in \mathcal{U}_{\eps}(u_{\omega})$. Then the assertion $(\textnormal{iii})$ follows by \eqref{dau}. Using the assertion $(\textnormal{i})$, we have that
$$
0=\frac{d}{d \theta} A(e^{\textnormal{i} \theta} u)\mid_{\theta=0}=\langle A'(u), \textnormal{i} u\rangle.
$$
Then the assertion $(\textnormal{iv})$ follows. Thus the proof is completed.
\end{proof}

\begin{lem}\label{ode}
Let $\mathcal{O}$ be an open set satisfying $\mathcal{O} \subset \mathcal{U}_{\eps}(u_{\omega})$ for $\eps>0$ small. Then there exist $s_0>0$ and a smoothing function $R : (-s_0, s_0) \times \mathcal{O}\to \mathcal{U}_{\eps}(u_{\omega})$ such that
\begin{itemize}
\item [$(\textnormal{i})$] $R(0, u)=u$,
%\item [$(\textnormal{ii})$] $e^{\textnormal{i} \theta} R(s, u)=R(s, e^{\textnormal{i} \theta} u)$,
\item [$(\textnormal{ii})$] $\frac{d}{ds} R(s, u) \mid_{s=0}=-\textnormal{i}A'(u)$,
\item [$(\textnormal{iii})$] $\|R(s, u)\|_2=\|R(0, u)\|_2$.
\end{itemize} 
\end{lem}
\begin{proof}
To prove this, we shall introduce the following differential equation,
\begin{align} \label{ode1}
\left\{
\begin{aligned}
\frac{dR}{ds}&=-\textnormal{i} A'(R),\\
R(0)&=u \in \mathcal{O}.
\end{aligned}
\right.
\end{align}
Due to Lemma \ref{propau}, then \eqref{ode1} is locally solved. Then there exists $s_0>0$ and a unique solution $R=R(s, u)$ to \eqref{ode1} for $|s| < s_0$. Clearly, the assertions $(\textnormal{i})$ and $(\textnormal{ii})$ are valid. Notice that, by the assertion $(\textnormal{iv})$ of Lemma \ref{propau},
$$
\frac{d}{ds} \|R(s, u)\|_2^2=2\langle -\textnormal{i} A'(R(s,u)), R(s, u)\rangle=0.
$$
Then the assertion $(\textnormal{iii})$ follows. This completes the proof.
%$$
%e^{\textnormal{i} \theta} R(0, u)=e^{\textnormal{i} \theta} u=R(0, e^{\textnormal{i} \theta} u).
%$$
%In addition, we see that
%$$
%\alpha(e^{\textnormal{i} \theta} u)=\alpha(u) -\theta.
%$$
%It follows that $\alpha'(e^{\textnormal{i}\theta} u)=\alpha'(u)$. Let $R$ be the solution to \eqref{ode1}, by \eqref{dau}, then
%$$
%\frac{d}{ds}\left(e^{\textnormal{i} \theta} R(s, u)\right)=-\textnormal{i} e^{\textnormal{i} \theta} A'(R(s, u))=
%$$
\end{proof}

\begin{lem} \label{coer}
There exist $\eps_0>0$ and a $C^1$ function $s$ defined on the set $\mathcal{U}_{\eps_0, \omega}:=\left\{u \in \mathcal{U}_{\eps_0}(u_{\omega}) : \|u\|_{2}=\|u_{\omega}\|_2\right\}$ 
such that
$$
J_{\omega}(u_{\omega})<J_{\omega}(u) + s(u) Q(u), \quad u \in \mathcal{U}_{\eps_0, \omega}, u \neq e^{\textnormal{i} t} u_{\omega},
$$
where
$$
Q(u)=\langle J_{\omega}'(u), -\textnormal{i} A'(u) \rangle.
$$
\end{lem}
\begin{proof}
Let $R(s, u)$ be the solution to \eqref{ode1}. Then we compute that
$$
\frac{\partial}{\partial s} J_{\omega}(R(s, u))\mid_{s=0}=\langle J_{\omega}'(u), -\textnormal{i}A'(u)\rangle=Q(u),
$$
$$
\frac{\partial^2}{\partial s^2} J_{\omega}(R(s, u))\mid_{s=0}=\langle J_{\omega}''(u) (\textnormal{i}A'(u)), -\textnormal{i}A'(u)\rangle-\langle J_{\omega}'(u) , \textnormal{i}A''(A'(u))\rangle.
$$
Moreover, we have that
$$
J_{\omega}((u_{\omega})_t)=\frac{t}{2} \int_{\R^2} |\partial_x u_{\omega}|^2 \, dxdy + \frac{t}{2} \int_{\R^2} |\partial_{yy} u_{\omega}|^2 \, dxdy + \frac 12 \int_{\R^2} |u_{\omega}|^2 \, dxdy-\frac{t^{\frac 38 (p-2)}}{p} \int_{\R^2} |u_{\omega}|^p \, dxdy.
$$
This readily shows that
$$
\frac{\partial^2}{\partial t^2} J_{\omega}((u_{\omega})_t) \mid_{t=1}=-\frac{3(p-2)(3p-14)}{64p} \int_{\R^2} |u_{\omega}|^p \, dxdy<0, \quad p>\frac{14}{3}.
$$
On the other hand, noting that $J_{\omega}'(u_{\omega})=0$, we then see that
$$
\frac{\partial^2}{\partial t^2} J_{\omega}((u_{\omega})_t) \mid_{t=1}=\langle J_{\omega}''(u_{\omega}) \psi_{\omega}, \psi_{\omega}\rangle.
$$
Therefore, we get that $\langle J_{\omega}''(u_{\omega}) \psi_{\omega}, \psi_{\omega}\rangle<0$. It then follows from the assertion $(\textnormal{ii})$ of Lemma \ref{propau} that
\begin{align*}
\frac{\partial^2}{\partial s^2} J_{\omega}(u_{\omega})&=\langle J_{\omega}''(u_{\omega}) (\textnormal{i}A'(u_{\omega})), -\textnormal{i}A'(u_{\omega})\rangle-\langle J_{\omega}'(u_{\omega}) , \textnormal{i}A''(A'(u_{\omega}))\rangle \\
&=\langle J_{\omega}''(u_{\omega}) \psi_{\omega}, \psi_{\omega}\rangle<0.
\end{align*}
Taking into account Taylor's expansion, we then have that
\begin{align} \label{te}
J_{\omega}(R(s, u))<J_{\omega}(u) + s Q(u)
\end{align}
for $|s|<s_0$ and $u\in B(u_{\omega}, \eps_0)$ for some $s_0>0$ and $\eps_0>0$ small. Note that
$$
I_{\omega}(R(s, u)) \mid_{(s,u)=(0, u_{\omega})}=0.
$$
Then there holds that
$$
\frac{\partial}{\partial s}I_{\omega}(R(s, u))\mid_{(s,u)=(0, u_{\omega})}=\langle I_{\omega}'(u_{\omega}), -\textnormal{i} A'(u_{\omega})\rangle=\langle I_{\omega}'(u_{\omega}), \psi_{\omega} \rangle=-\frac{3(p-2)^2}{8p} \int_{\R^2}|u_{\omega}|^p \,dxdy<0.
$$
As a consequence of the implicit function theorem, then there exists a $C^1$ function $s : B(u_{\omega}, \eps_0) \to \R$ such that $I_{\omega}(R(s(u), u))=0$ for some $\eps_0>0$ smaller if necessary. This shows that
$$
J_{\omega} (u_{\omega})\leq J_{\omega}(R(s(u), u)).
$$
Utilizing \eqref{te}, we then derive the desired conclusion. Thus the proof is completed.
\end{proof}

\begin{lem} \label{invariant}
Define
$$
\mathcal{S}^+:=\left\{u \in \mathcal{O} : J_{\omega}(u)<m_{\omega}, Q(u)>0\right\},
$$
$$
\mathcal{S}^-:=\left\{u \in \mathcal{O} : J_{\omega}(u)<m_{\omega}, Q(u)<0\right\}.
$$
Then the sets $\mathcal{S}^+$ and $\mathcal{S}^-$ are non-empty invariant under the flow of the Cauchy problem for \eqref{equt}. Moreover, for any $\psi_0 \in \mathcal{S}^+ \cup \mathcal{S}^-$, there exists $\delta_0>0$ such that $|Q(\psi(t))|\geq \delta_0$ for any $t \in [0, T_{\eps}(\psi_0))$, where
$$
T_{\eps}(\psi_0):=\sup \left\{ T>0 : \psi(t) \in \mathcal{U}_{\eps}(u_{\omega}) \,\, \mbox{for any} \,\, t \in [0, T) \right\}.
$$
\end{lem}
\begin{proof}
Define $(u_{\omega})_{\tau}:=\tau^{\frac 38} u_{\omega}(\tau^{\frac 12} \cdot, \tau^{\frac 14} \cdot)$ for any $\tau>0$. Since $Q(u_{\omega})=0$, then $J_{\omega}((u_{\omega})_{\tau})<J_{\omega}(u_{\omega})$ for any $\tau>0$. Furthermore, we see that $Q((u_{\omega})_{\tau})>0$ if $0<\tau<1$ and $Q((u_{\omega})_{\tau})<0$ if $\tau>1$. This justifies that $\mathcal{S}^+$ and $\mathcal{S}^-$ are non-empty. Next we show that $\mathcal{S}^+$ and $\mathcal{S}^-$ are invariant under the flow of the Cauchy problem for \eqref{equt}. Let $\psi_0 \in \mathcal{S}^+$, by the conservation laws, then, for any $t \in [0, T_{\eps}(\psi_0))$,
\begin{align} \label{laws}
J_{\omega}(\psi(t))=J_{\omega}(\psi_0)<m_{\omega}.
\end{align}
If there exists some $t_0 \in (0, T_{\eps}(\psi_0))$ such that $Q(\psi(t_0)) \leq 0$, by the continuity, then there exists $t_1 \in (0, T_{\eps}(\psi_0))$ such that $Q(\psi(t_1))=0$, because of $Q(\psi_0)>0$. This in turn gives that $m_{\omega} \leq J_{\omega}(\psi(t_1))$. This is impossible, due to \eqref{laws}. It then follows that $Q(\phi(t))>0$ for any $t \in [0, T_{\eps}(\phi_0))$. Using Lemma \ref{coer}, we then obtain that
$$
0<m_{\omega}- J_{\omega}(\psi_0)=J_{\omega}(u_{\omega}) - J_{\omega}(\psi(t))<s Q(\psi(t)) \leq C |Q(\psi(t))|.
$$
The similar results holds for $\psi_0 \in \mathcal{S}^-$. Thus the proof is completed.
\end{proof}

\begin{proof}[Proof of Theorem \ref{instability}]
Let $\psi_0 \in \mathcal{S}^+ \cup \mathcal{S}^-$ and $\psi \in C([0, T); H^{1,2}(\R^2))$ be the solution to \eqref{equt} with initial datum $\psi_0$. Suppose by contradiction that $T(\psi_0)=+\infty$. Therefore, from Lemma \ref{invariant}, we have that $|Q(\psi(t))| \geq \delta_0$ for any $t \geq 0$. Observe that
$$
\frac{d}{dt}A(\psi(t))=\langle A'(\psi), -\textnormal{i} E'(\psi) \rangle=\langle A'(\psi), -\textnormal{i} J_{\omega}'(\psi) \rangle=Q(\psi(t)),
$$
where we used the assertion $(\textnormal{iv})$ of Lemma \ref{propau}. This infers that $|A(\psi(t))| \to +\infty$ as $t \to +\infty$. However, by the definition, we know that $|A(u)| \leq C$ for any $u \in \mathcal{U}_{\eps}(u_{\omega})$. Thus we have that $T_{\eps}(\psi_0)<+\infty$. This then implies the desired result and the proof is completed.
\end{proof}

\section{Global existence and blowup of solutions} \label{dynamics}

In this section, we are going to consider the global well-posedness and  blowup of solutions to the Cauchy problem for \eqref{equ} and present the proofs of Theorems \ref{gw} and \ref{blowup}.

\begin{proof}[Proof of Theorem \ref{gw}] 
Let $u \in C([0, T); H^{1,2}(\R^2))$ be the solution to the Cauchy problem for \eqref{equt} with initial datum $u_0 \in H^{1,2}(\R^2)$. In view of \eqref{gn}, then
\begin{align} \label{gl}
\begin{split}
E(u(t))&\geq \frac 12 \int_{\R^2} |\partial_x u(t)|^2 \,dxdy +\frac 12 \int_{\R^2} |\partial_{yy} u(t)|^2 \,dxdy \\
& \quad - \frac{C_{opt}}{p} \left(\int_{\R^2} |\partial_x u(t)|^2 \,dxdy +\int_{\R^2} |\partial_{yy} u(t)|^2 \,dxdy\right)^{\frac{3(p-2)}{8}} \left(\int_{\R^2} |u(t)|^2 \,dxdy\right)^{\frac{p+6}{8}}.
\end{split}
\end{align}
If $2<p<\frac{14}{3}$, then $0<\frac{3(p-2)}{8}<1$. It then follows from \eqref{gl} and the conservation laws that, for any $t \in [0, T)$,
$$
 \int_{\R^2} |\partial_x u(t)|^2 \,dxdy + \int_{\R^2} |\partial_{yy} u(t)|^2 \,dxdy \leq C.
$$
Then there holds that $T=+\infty$. If $p=\frac{14}{3}$ and $\|u(t)\|_2 < c_*$, then the same conclusion holds true. If $p>\frac{14}{3}$, arguing as the proof of Lemma \ref{invariant}, we can similarly obtain that $\mathcal{G}$ is invariant under the flow of the Cauchy problem for \eqref{equt}. Let us suppose that $u(t)$ blows up in finite time $T<\infty$. Observe that
$$
E(u(t))-\frac{4}{3(p-2)} Q(u(t))=\frac{3p-14}{6(p-2)} \left(\int_{\R^2} |\partial_x u(t)|^2\, dxdy+\int_{\R^2} |\partial_{yy} u(t)|^2\, dxdy\right).
$$
Using the conservation of energy, we then have that $\lim_{t \to T^-} Q(u(t))=-\infty$. By the continuity, then there exists $0<t_0<T$ such that $Q(u(t_0))=0$. This along with the conservation laws leads to
$$
J_{\omega}(u(t_0))<m_{\omega} \leq J_{\omega}(u(t_0))=J_{\omega}(u(t)).
$$
This is impossible. Then there holds that $T=+\infty$ and the proof is completed.
\end{proof}

In the follows, we shall discuss blowup of solutions to the Cauchy problem for \eqref{equt}. 
%Let us first introduce a functional $K : H^{1,2}(\R^2) \to \R$ as
%$$
%K(u):=\frac 12 \int_{\R^2} |\partial_{yy} u|^2\, dxdy-\frac{p-2}{8p} \int_{\R^2}|u|^p\,dxdy.
%$$ 
To establish Theorem \ref{blowup}, we first present variational characterization of the ground state energy level $m_{\omega}$.

%\begin{lem} \label{vc2}
%Let $p>2$ and $u_{\omega} \in H^{1,2}(\R^2)$ be a ground state to \eqref{equ}. Then
%$$
%J_{\omega}(u_{\omega})=\inf \left\{J_{\omega}(u) : u \in H^{1,2}(\R^2)\backslash \{0\}, \|u\|_p=\|u_{\omega}\|_p \right\}.
%$$
%\end{lem}
%\begin{proof}
%Observe that
%\begin{align*}
%&\inf \left\{J_{\omega}(u) : u \in H^{1,2}(\R^2)\backslash \{0\}, \|u\|_p=\|u_{\omega}\|_p \right\} \\
%&=\frac 12 \inf \left\{ \|\partial_x u\|_2^2 + \|\partial_{yy} u\|_2^2 + \omega \|u\|_2^2 : u \in H^{1,2}(\R^2)\backslash \{0\}, \|u\|_p=\|u_{\omega}\|_p \right\}-\frac 1p \|u_{\omega}\|_p^p =J_{\omega}(u_{\omega}).
%\end{align*}
%Indeed, by Lemma \ref{ccl} and the Lions concentration compactness principle in \cite{Li1,Li2}, one can easily show that the following minimization problem is achieved by $u_{\omega}$,
%\begin{align} \label{minlp}
%\inf \left\{ \|\partial_x u\|_2^2 + \|\partial_{yy} u\|_2^2 + \omega \|u\|_2^2 : u \in H^{1,2}(\R^2)\backslash \{0\}, \|u\|_p=\|u_{\omega}\|_p \right\}.
%\end{align}
%Thus the proof is completed.
%\end{proof}

\begin{lem} \label{vc0}
Let $p>2$, $\omega>0$ and $u_{\omega} \in H^{1,2}(\R^2)$ be a ground state to \eqref{equ}. Then $K(u_{\omega})=0$. Moreover, if $p > 10$, then there holds that
\begin{align} \label{minyy}
%\|\partial_{yy} u_{\omega}\|_2=\inf \left\{\|\partial_{yy} u\|_2 : u \in H^{1,2}(\R^2) \backslash \{0\}, K_{\omega}(u)=0\right\}.
J_{\omega}(u_{\omega})=\inf \left\{J_{\omega}(u) : u \in H^{1,2}(\R^2) \backslash \{0\}, K(u)=0\right\}.
\end{align}
\end{lem}
\begin{proof}
First we prove that $K(u_{\omega})=0$. Define $u^{\lambda}:=\lambda^{\frac 12} u(\cdot, \lambda \cdot)$ for any $u \in H^{1,2}(\R^2)$ and $\lambda>0$. It is straightforward to compute that
\begin{align*}% \label{scaling}
J_{\omega}(u^{\lambda})=\frac{1}{2} \int_{\R^2} |\partial_x u|^2 \,dxdy + \frac{\lambda^4}{2} \int_{\R^2} |\partial_{yy} u|^2 \,dxdy+ \frac{\omega}{2} \int_{\R^2} |u|^2 \, dxdy -\frac{\lambda^{\frac p2-1}}{p} \int_{\R^2} |u|^p \,dxdy.
\end{align*}
Observe that
$$
\frac{d}{d \lambda} J_{\omega}((u_{\omega})^{\lambda})\mid_{\lambda=1}=\left\langle J'_{\omega}(u_{\omega}), \frac{d}{d\lambda}(u_{\omega})^{\lambda}\mid_{\lambda=1} \right\rangle=0.
$$
It then follows that
\begin{align} \label{iyy}
\int_{\R^2} |\partial_{yy} u_{\omega}|^2\, dxdy=\frac{p-2}{4p} \int_{\R^2}|u_{\omega}|^p\,dxdy.
\end{align}
%On the other hand, since $u_{\omega} \in H^{1,2}(\R^2)$ is a solution to \eqref{equ}, then $I_{\omega}(u_{\omega})=0$. Therefore, we derive that
%$$
%J_{\omega}(u_{\omega})=\frac{p-2}{2p} \int_{\R^2}|u_{\omega}|^p\, dxdy.
%$$
%This along with \eqref{iyy} yields that
%$$
%J_{\omega}(u_{\omega})=2\int_{\R^2} |\partial_{yy} u_{\omega}|^2\, dxdy=\frac{p-2}{2p} \int_{\R^2}|u_{\omega}|^p\, dxdy.
%$$
%Since $Q(u_{\omega})=0$ by Lemma \ref{ph}, then
%\begin{align} \label{iyy1}
%\int_{\R^2} |\partial_x u_{\omega}|^2\, dxdy=\frac{p-2}{2p} \int_{\R^2}|u_{\omega}|^p\, dxdy.
%\end{align}
%In addition, since $I_{\omega}(u_{\omega})=0$, then
%\begin{align} \label{iyy2}
%\omega \int_{\R^2} |u_{\omega}|^2\, dxdy=\frac{p+6}{4p} \int_{\R^2}|u_{\omega}|^p \, dxdy.
%\end{align}
%Making use of \eqref{iyy}, \eqref{iyy1} and \eqref{iyy2}, we then derive that $K_{\omega}(u_{\omega})=0$. 
This obviously shows that $K(u_{\omega})=0$. Now we demonstrate that \eqref{minyy} holds true for $p>10$. Since $K(u_{\omega})=0$, then
$$
%\|\partial_{yy} u_{\omega}\|_2 \geq \inf \left\{\|\partial_{yy} u\|_2 : u \in H^{1,2}(\R^2) \backslash \{0\}, K_{\omega}(u)=0\right\}.
J_{\omega}(u_{\omega}) \geq \inf \left\{J_{\omega}(u) : u \in H^{1,2}(\R^2) \backslash \{0\}, K_{\omega}(u)=0\right\}.
$$
Let $u \in H^{1,2}(\R^2) \backslash \{0\}$ and $K(u)=0$. Note that 
%$\lim_{\lambda \to 0^+}J_{\omega}(u^{\lambda})>0$ and 
$\lim_{\lambda \to +\infty}J_{\omega}(u^{\lambda})=-\infty$,
%$J_{\omega}(u^{\lambda}) >0$ for any $0<\lambda<1$ small and $J_{\omega}(u^{\lambda})<0$ for any $\lambda>1$,
because of $p>10$. Furthermore, there holds that $\max_{\lambda>0}J_{\omega}(u^{\lambda})=J_{\omega}(u)$. Define a path $\gamma: [0, 1] \to H^{1,2}(\R^2)$ by 
\begin{align*}
\gamma(t):=\left\{
\begin{aligned}
&t\lambda_0 u, \quad & 0 \leq t\leq \lambda_0^{-1},\\
&u^{t \lambda_0}, \quad &\lambda_0^{-1} \leq t \leq 1.
\end{aligned}
\right.
\end{align*}
where $\lambda_0>1$ sufficently large such that $J_{\omega}(u^{\lambda_0})<0$. Clearly, there holds that $\lim_{t \to 0^+}J_{\omega}(\gamma(t))=0$ and $\lim_{t \to 1^-}J_{\omega}(\gamma(t))<0$. As a consequence, we are able to derive that
$$
J_{\omega}(u_{\omega})=m_{\omega} \leq \max_{t \in [0, 1]}J_{\omega}(\gamma(t)) = J_{\omega}(u).
$$
Therefore, we conclude that
$$
J_{\omega}(u_{\omega}) \leq \inf \left\{J_{\omega}(u) : u \in H^{1,2}(\R^2) \backslash \{0\}, K(u)=0\right\}.
$$
This completes the proof.
\end{proof}

To discuss blowup of solutions to the Cauchy problem for \eqref{equt}, we need to introduce the following transverse virial quantity,
$$
V[u]:=\int_{\R^2} |x|^2 |u|^2 \,dxdy.
$$

\begin{lem} \label{virial}
Let $p>2$ and $u_0 \in H^{1,2}(\R^2)$ be such that $x u_0 \in L^2(\R^2)$. Let $u \in C([0, T), H^{1,2}(\R^2))$ be the solution to the Cauchy problem for \eqref{equt} with initial datum $u_0$. Then there holds that, for any $0 < t<T$,
$$
\frac{d^2}{dt^2}V[u(t)]=8\int_{\R^2} |\partial_x u(t)|^2 \, dxdy-\frac{4(p-2)}{p} \int_{\R^2}|u(t)|^p \, dxdy.
$$
\end{lem}
\begin{proof}
First it is simple to compute that
\begin{align*}
\frac{d}{dt} V[u(t)]&=2\Re \int_{\R^2} |x|^2u_t \overline{u} \, dxdy=2 \Im \int_{\R^2} |x|^2 \textnormal{i}u_t \overline{u} \, dxdy\\
&=2 \Im \int_{\R^2} |x|^2 \left(-\partial_{xx} u + \partial_{yyyy} u-|u|^{p-2} u\right)\overline{u} \, dxdy\\
&=4 \Im \int_{\R^2} x \left(\partial_x u \right) \overline{u} \, dxdy.
\end{align*}
Further, we see that
\begin{align*}
\frac{d^2}{dt^2} V[u(t)]&=4 \Im \int_{\R^2} x \left(\partial_x u_t \right) \overline{u} \, dxdy +4 \Im \int_{\R^2} x \left(\partial_x u \right) \overline{u}_t \, dxdy \\
&=-4\Re \int_{\R^2} x \left(\partial_x \textnormal{i} u_t \right) \overline{u} \, dxdy+4 \Re \int_{\R^2} x \left(\partial_x \overline{u} \right) \textnormal{i} u_t \, dxdy \\
&=-4 \Re \int_{\R^2} x \left(\partial_x \left(-\partial_{xx} u + \partial_{yyyy} u-|u|^{p-2} u\right) \right) \overline{u} \, dxdy \\
& \quad +4 \Re \int_{\R^2} x \left(\partial_x \overline{u} \right)\left(-\partial_{xx} u + \partial_{yyyy} u-|u|^{p-2} u\right)  \, dxdy \\
&= 4\int_{\R^2} |\partial_x u|^2 +|\partial_{yy} u|^2 -|u|^{p} \, dxdy + 8 \Re \int_{\R^2} x \left(\partial_x \overline{u} \right)\left(-\partial_{xx} u + \partial_{yyyy} u-|u|^{p-2} u\right)  \, dxdy  \\
&=8\int_{\R^2} |\partial_x u|^2 \, dxdy-\frac{4(p-2)}{p} \int_{\R^2}|u|^p \, dxdy.
\end{align*}
Thus the proof is completed.
\end{proof}
 
\begin{proof} [Proof of Theorem \ref{blowup}]
For simplicity, we shall write $u=u(t)$ as the solution to the Cauchy problem for \eqref{equt} with initial datum $u_0$. Let us first deal with the case that $E(u_0) < 0$ and $p \geq 6$. In this case, by Lemma \ref{virial}, then
\begin{align} \label{sdv0}
\frac{d^2}{dt^2} V[u]=16 E(u)-8 \int_{\R^2} |\partial_{yy} u|^2 \, dxdy-\frac{4(p-6)}{p} \int_{\R^2}|u|^p \, dxdy.
\end{align}
Due to $p \geq 6$, by \eqref{sdv0} and the conservation of energy, then
$$
\frac{d^2}{dt^2} V[u] \leq 16 E(u) =16E(u_0)<0.
$$
This implies that $u$ cannot exist globally in time. Next we consider the case that $u_0 \in \mathcal{B}$ and $p>10$. In the case, with the help of the conservation laws, arguing as the proof in Lemma \ref{invariant}, we are able to show that $J_{\omega}(u)<m_{\omega}$ and $Q(u)<0$. In addition, using Lemma \ref{vc0}, we can also prove that $J_{\omega}(u)<m_{\omega}$ and $K(u)>0$. Consequently, we get that $\mathcal{B}$ is invariant under the flow of the Cauchy problem for \eqref{equt}. In view of Lemma \ref{virial}, then
\begin{align} \label{sdv}
\begin{split}
\frac{d^2}{dt^2} V[u]&=8 Q(u) -8\int_{\R^2} |\partial_{yy} u|^2 \, dxdy +\frac{2(p-2)}{p} \int_{\R^2}|u|^p \, dxdy<8Q(u).
\end{split}
\end{align}
Since $Q(u)<0$, then there exists $0<\tau_u<1$ such that $Q(u_{\tau_u})=0$, where $u_{\tau}:=\tau^{\frac 38} u(\tau^{\frac 12} \cdot, \tau^{\frac 14} \cdot)$ for $\tau>0$. It is easy to calculate that the function $\tau \mapsto J_{\omega}(u_{\tau})$ is concave on $[\tau_u, +\infty)$. Then we have that
$$
J_{\omega} (u_0)=J_{\omega}(u) \geq J_{\omega} (u_{\tau_{u}}) + (1-\tau_{u}) \frac{d}{d\tau} J_{\omega}(u_{\tau}) \mid_{\tau=1} \geq J_{\omega}(u_{\omega}) +(1-\tau_{u}) Q(u) >J_{\omega}(u_{\omega}) +Q(u),
$$
from which there holds that $Q(u) <J_{\omega}(u_0) -m_{\omega}<0$. Going back to \eqref{sdv}, we then have that $u$ cannot exist globally in time. This completes the proof.
\end{proof}

\subsection*{Data Availability}
 Data sharing is not applicable to this article as no data sets were generated or analysed during the current study.

\end{document}